\documentclass{amsart}%
\usepackage{amsfonts}
\usepackage{amsmath}
\usepackage{amssymb}
\usepackage{graphicx}
\usepackage[latin1]{inputenc}
\usepackage[english]{babel}%
\newtheorem{theorem}{Theorem}[section]
\theoremstyle{plain}

\newtheorem{corollary}[theorem]{Corollary}

\newtheorem{definition}[theorem]{Definition}
\newtheorem{example}[theorem]{Example}

\newtheorem{lemma}[theorem]{Lemma}

\newtheorem{proposition}[theorem]{Proposition}
\newtheorem{remark}[theorem]{Remark}

\numberwithin{equation}{section}
\begin{document}
\title[Zeta Functions for Laurent Polynomials]{Local Zeta Functions for Non-degenerate Laurent Polynomials Over p-adic Fields}
\author{E. León-Cardenal}
\curraddr{Centro de Ciencias Matemáticas\\
UNAM, Campus Morelia \\
Km. 8 Antigua Carretera a Pátzcuaro \#8701. Col. Ex-hacienda San José de la
Huerta. Morelia, Michoacán. Mexico.}
\email{edwin@matmor.unam.mx}
\author{W. A. Zúñiga-Galindo}
\address{Centro de Investigacion y de Estudios Avanzados del I.P.N., Departamento de
Matematicas, Av. Instituto Politecnico Nacional 2508, Col. San Pedro
Zacatenco, Mexico D.F., C.P. 07360, Mexico}
\email{wazuniga@math.cinvestav.edu.mx}
\thanks{The second author was partially supported by Conacyt (Mexico), Grant \# 127794.}

\begin{abstract}
In this article, we study local zeta functions attached to Laurent polynomials
over $p$-adic fields, which are non-degenerate with respect to their Newton
polytopes at infinity. As an application we obtain asymptotic expansions for
$p$-adic oscillatory integrals attached to Laurent polynomials. We show the
existence of two di\-fferent asymptotic expansions for $p$-adic oscillatory
integrals, one when the absolute value of the parameter approaches infinity,
the other when the absolute value of the parameter approaches zero. These two
asymptotic expansions are controlled by the poles of twisted local zeta
functions of Igusa type.

\end{abstract}
\maketitle

\section{Introduction}

The local zeta functions for non-degenerate polynomials (or more generally for
non-degenerate analytic functions) have been studied quite extensively.
Initially these functions were studied by Varchenko in the Archimedean case,
later Denef studied them in the non-Archimedean case, see e.g. \cite{AVG},
\cite{D3}, \cite{D-H}, \cite{D-Sp}, \cite{L-M}, \cite{Var1}, \cite{V-Z},
\cite{Z1}, \cite{Z2}, among others. In this article, we study local zeta
functions attached to Laurent polynomials over $p$-adic fields, which are
weakly non-degenerate with respect to their Newton polytopes at infinity, see
Definition \ref{nondegerate}. This notion of non-degeneracy is weaker than the
standard non-degeneracy condition of Khovanskii, see Definition
\ref{classical} and Example \ref{Example1}. By using a variation of to\-ric
resolution of singularities, we show the existence of a meromorphic
continuation for these zeta functions as rational functions of $q^{-s}$, see
Theorem \ref{Th1}. We also extend Igusa's stationary phase method for
oscillatory integrals (and certain exponential sums) depending on a $p$-adic
parameter to the case of Laurent polynomials, see Theorem \ref{Th2}. Here, a
new and interesting phenomenon occurs: there are two different asymptotic
expansions for $p$-adic oscillatory integrals, one when the absolute value of
the parameter approaches infinity, the other when the absolute value of the
parameter approaches zero. These two asymptotic expansions are controlled by
the poles of twisted local zeta functions.

The classical local zeta functions are connected with polynomial congruences
$\operatorname{mod}$ $p^{m}$. In the case of Laurent \ polynomials the
corresponding local zeta functions control the asymptotic behavior of the
volumes of `tubular neighborhoods' attached to the polynomials, see Theorem
\ref{Th1AA}.

There are several important differences between the classical local zeta
functions for non-degenerate polynomials and the local zeta functions studied
here. First, the classical local zeta functions have only poles with negative
real parts while the local zeta functions for Laurent polynomials have poles
with positive and negative real parts. This fact makes more difficult the
determination of the actual poles of these new local zeta functions. Second,
the convergence of the integral defining the local zeta function (see
Definition \ref{zeta_function}) is not a straightforward matter due to the
presence of `denominators.'

Finally we want to comment that our initial motivation was to find $p$-adic
counterparts of certain estimates for exponential sums attached to
non-degenerate Laurent polynomials over finite fields due to Adolphson and
Sperber \cite{A-S2} and Denef and Loeser \cite{D-L}, see Corollary \ref{cor1}.

\textbf{Acknowledgement. }The authors want to thank to the referees for their
careful reading of the article and for several useful suggestions.

\section{\label{Sec1}Newton Polytopes, Non-degeneracy Conditions and Toric
Manifolds}

In this section, we review some basic results on toric manifolds, and
non-degeneracy conditions for Laurent polynomials over a local field of
characteristic zero. The results needed here are variations of the ones given
in \cite{KH1}-\cite{KH2}, \cite{O}, in the Archimedean setting. The material
needed to adapt these results to the $p$-adic setting can be found in
\cite{I2}, \cite{Ser}.

\subsection{Newton Polytopes}

We set $\mathbb{R}_{+}:=\{x\in\mathbb{R};x\geqslant0\}$. Let $\left\langle
\cdot,\cdot\right\rangle $ denote the usual inner product of $\mathbb{R}^{n}$,
and identify \ the dual space of $\mathbb{R}^{n}$ with $\mathbb{R}^{n}$ itself
by means of it.

Let $K$ be a local field of characteristic zero. Let
\[
f(x)=\sum_{m\in\mathbb{Z}^{n}}a_{m}x^{m}\in K\left[  x_{1},\ldots,x_{n}%
,x_{1}^{-1},\ldots,x_{n}^{-1}\right]
\]
be a non-constant Laurent polynomial. Set $supp(f):=\left\{  m\in
\mathbb{Z}^{n};a_{m}\neq0\right\}  $. We define \textit{the Newton polytope
\ }$\Gamma_{\infty}\left(  f\right)  :=\Gamma_{\infty}$\textit{\ of }%
$f$\textit{\ at infinity} as the convex hull of $supp(f)$ in $\mathbb{R}^{n}$.
Note that, if $supp(f)=\left\{  m_{1},\ldots,m_{l}\right\}  $, then
\[
\Gamma_{\infty}=conv\left(  m_{1},\ldots,m_{l}\right)  =\left\{
{\displaystyle\sum\limits_{i=1}^{l}}
\lambda_{i}m_{i};\lambda_{1},\ldots,\lambda_{l}\in\mathbb{R}_{+}\text{, }%
{\displaystyle\sum\limits_{i=1}^{l}}
\lambda_{i}=1\right\}  .
\]
In combinatorics a set like $\Gamma_{\infty}$ is typically called
a\textit{\ rational (or lattice) polytope (i.e. a compact polyhedron)}. From
now on, we will use just \textit{polytope} to mean \textit{rational polytope}
and \textit{assume that }$\dim\Gamma_{\infty}=n$.

\subsubsection{Faces}

Let $H$ be the hyperplane $\left\{  x\in\mathbb{R}^{n};\left\langle
a,x\right\rangle =b\right\}  $. Then $H$ determines two closed half-spaces:%
\[
H^{+}:=\left\{  x\in\mathbb{R}^{n};\left\langle a,x\right\rangle \geq
b\right\}
\]
and
\[
H^{-}:=\left\{  x\in\mathbb{R}^{n};\left\langle a,x\right\rangle \leq
b\right\}  .
\]
We say that $H$ is \textit{a supporting hyperplane} of $\Gamma_{\infty}$, if
$\Gamma_{\infty}\cap H\neq\emptyset$ and $\Gamma_{\infty}$ is contained in one
of the closed half-spaces determined by $H$.

The \textit{dimension of a face }$\tau$ of $\Gamma_{\infty}$ is the dimension
of its affine span, and its \textit{codimension} is $cod\left(  \tau\right)
=n-\dim\left(  \tau\right)  $. A face of codimension $1$ is a \textit{facet}.
Faces of dimension $0$ and $1$ are called \textit{vertices} and \textit{edges}
respectively. We denote by $vert(\Gamma_{\infty})$ the set of vertices of
\ $\Gamma_{\infty}$. A face of $\Gamma_{\infty}$ different from $\Gamma
_{\infty}$ is called\textit{ proper.}

Given $a\in\mathbb{R}^{n}$, we define
\[
d(a,\Gamma_{\infty}):=d(a)=\inf\left\{  \left\langle a,x\right\rangle
;x\in\Gamma_{\infty}\right\}  .
\]
Note, that since a convex polytope is the convex hull of its vertices, we can
take the infimum as $v$ varies in $vert(\Gamma_{\infty})$, which is a finite
set, hence
\[
d(a)=\min\left\{  \left\langle a,x\right\rangle ;x\in vert(\Gamma_{\infty
})\right\}  ,
\]
and $d(a)=\left\langle a,x_{0}\right\rangle $ for some $x_{0}\in
vert(\Gamma_{\infty})$.

\subsubsection{Primitive vectors and facets}

Given a supporting hyperplane $H$ containing a facet of $\Gamma_{\infty}$,
there exists a unique vector $a\in\mathbb{Z}^{n}\smallsetminus\left\{
0\right\}  $ perpendicular to $H$ and directed into the polytope. This vector
is called the \textit{inward} normal to $H$. A vector $a=(a_{1},\ldots
,a_{n})\in\mathbb{Z}^{n}$ is called primitive if $g.c.d.(a_{1},\ldots
,a_{n})=1$. Every facet has a unique primitive inward vector. We denote the
set of all these vectors as $\mathfrak{D}(\Gamma_{\infty})$.

\subsection{Cones and Fans}

We now review the construction of conical subdivisions of $\mathbb{R}^{n}$ and
\ $\mathbb{R}_{+}^{n}$ \textit{subordinated to} $\Gamma_{\infty}$. Such
constructions are simple variation of some well-known ones, see e.g.
\cite{Ew}, \cite[The main example, Section 1.2]{KH2}, \cite[Chapter 7]{Zie},
we also use \cite{D-H}, \cite{Var1}, for this reason we do not give proofs.

We recall that the\textit{ cone} \textit{strictly\ spanned }\ by the vectors
$a_{1},\ldots,a_{r}\in\mathbb{R}^{n}\setminus\left\{  0\right\}  $ is the set
$\Delta^{\circ}=\left\{  \lambda_{1}a_{1}+...+\lambda_{r}a_{r};\lambda_{i}%
\in\mathbb{R}_{+}\text{, }\lambda_{i}>0\right\}  $. Notice that the
topological closure of $\Delta^{\circ}$ is
\begin{equation}
\Delta=\left\{  \lambda_{1}a_{1}+...+\lambda_{r}a_{r};\lambda_{i}\in
\mathbb{R}_{+}\text{, }\lambda_{i}\geq0\right\}  , \label{cone}%
\end{equation}
cf. Lemma \ref{lema1}. This set is typically called a \textit{convex
polyhedral cone}. If $a_{1},\ldots,a_{r}$ are linearly independent over
$\mathbb{R}$, $\Delta^{\circ}$ and $\Delta$ are called \textit{simplicial
cones}. If \ $a_{1},\ldots,a_{r}\in\mathbb{Z}^{n}$, we say $\Delta^{\circ}$
and $\Delta$\ are \textit{rational cones}. If $\left\{  a_{1},\ldots
,a_{r}\right\}  $ is a subset of a basis \ of the $\mathbb{Z}$-module
$\mathbb{Z}^{n}$, we call $\Delta^{\circ}$ and $\Delta$ \textit{simple cones}.
The justification for this `unusual' approach is the following. The
computation and the obtention of explicit formulas of local zeta functions
require `open cones', see Remark \ref{Nota2AA} (ii) and \cite{D-H},
\cite{V-Z}, \cite{Z2}, while calculations using toroidal resolution of
singularities require `closed cones'.

We define \textit{the first meet locus} of $a\in\mathbb{R}^{n}$ as%
\[
F(a,\Gamma_{\infty}):=F(a)=\left\{  x\in\Gamma_{\infty};\left\langle
a,x\right\rangle =d\left(  a\right)  \right\}  .
\]
Note that $F(a)$\ is a face of $\Gamma_{\infty}$, and that $F(0)=\Gamma
_{\infty}$.

We define an equivalence relation on $\mathbb{R}^{n}$\ by taking%
\[
a\sim a^{\prime}\Longleftrightarrow F\left(  a\right)  =F\left(  a^{\prime
}\right)  .
\]
If $\tau$ is a face of $\Gamma_{\infty}$, we define \textit{the cone
associated to} $\tau$\ as
\[
\Delta_{\tau}^{\circ}=\left\{  a\in\mathbb{R}_{+}^{n};F\left(  a\right)
=\tau\right\}  .
\]

\begin{lemma}
\label{lema1}Let $\tau$ be a proper face of $\Gamma_{\infty}$. Then

\noindent(1) $\Delta_{\tau}^{\circ}$ is a relatively open in the vector
subspace of $\mathbb{R}^{n}$ spanned by $\Delta_{\tau}^{\circ}$.

\noindent(2) The topological closure $\Delta_{\tau}$ of $\Delta_{\tau}^{\circ
}$ is a rational convex polyhedral cone with vertex at the origin, and
\[
\Delta_{\tau}=\left\{  a\in\mathbb{R}^{n};F\left(  a\right)  \supset
\tau\right\}  .
\]
\noindent(3) $\dim\Delta_{\tau}^{\circ}=\dim\Delta_{\tau}=n-\dim\tau.$

\noindent(4) The function $d\left(  \cdot\right)  $ is linear on $\Delta
_{\tau}$.
\end{lemma}

We recall that a \textit{rational strongly convex polyhedral cone} $\Delta$ is
cone of form (\ref{cone}) with vertex at the origin of $\mathbb{R}^{n}$, and
with $a_{1},\ldots,a_{r}\in\mathbb{Z}^{n}$. It is also useful to recall that
$\Delta$ is the solution set of a system of inequalities of the form $Ax\leq
0$, where $A$ is a matrix with integer entries and $x\in\mathbb{R}^{n}$.

We recall that a \textit{fan} $\mathcal{L}$\ is a finite collection of
\textit{rational strongly convex polyhedral cones} $\{\Delta_{i}\mathbf{;}i\in
I\}$ in $\mathbb{R}^{n}$ such that: (i) if $\Delta_{i}\in\mathcal{L}$\ and
$\Delta$ is a face of $\Delta_{i}$, then $\Delta\in\mathcal{L}$; (ii) if
$\Delta_{1}$, $\Delta_{2}\in\mathcal{L}$, then $\Delta_{1}\cap$ $\Delta_{2}$
is a face of $\Delta_{1}$ and $\Delta_{2}$. The \textit{support} of
$\mathcal{L}$ is $\left\vert \mathcal{L}\right\vert :=\cup_{i\in I}\Delta_{i}%
$. A fan $\mathcal{L}$ is called \textit{simplicial} (resp. \textit{simple })
if all its cones are simplicial (resp. simple). A fan $\mathcal{L}$ is called
\textit{subordinated to} $\Gamma_{\infty}$, if every cone in $\mathcal{L}$ is
contained in an equivalence class of $\thicksim$. We denote by
$edges(\mathcal{F})$, the set of all edges (generators) of the cones in
$\mathcal{L}$.

\begin{lemma}
\label{lema2A}The closures $\Delta_{\tau}$ of the cones associated to the
faces of $\Gamma_{\infty}$ form a simplicial fan $\mathcal{F}$ subordinated to
$\Gamma_{\infty}$. Moreover, we have the following:

\noindent(i) Let $\tau$ be a proper face of $\Gamma_{\infty}$. Then the map%
\[%
\begin{array}
[c]{ccc}%
\left\{  \text{faces of }\Gamma_{\infty}\text{ that contain }\tau\right\}  &
\rightarrow & \left\{  \text{non-empty faces of }\Delta_{\tau}\right\} \\
&  & \\
\sigma & \rightarrow & \Delta_{\sigma}%
\end{array}
\]
is one-to-one and onto.

\noindent(ii) Let $\tau_{1}$, $\tau_{2}$ be faces of $\Gamma_{\infty}$.
Suppose that $\tau_{1}$ is a facet of $\tau_{2}$, i.e. $\tau_{1}$ has
codimension one in $\tau_{2}$, then $\Delta_{\tau_{2}}$ is a\ facet of
$\Delta_{\tau_{1}}$.
\end{lemma}

\begin{lemma}
\label{lema2}(i) Let $\tau$ be a proper face of $\Gamma_{\infty}$. Let
$\gamma_{1},\ldots,\gamma_{r}$ be the facets of $\Gamma_{\infty}$ containing
$\tau$. Let $a_{1},\ldots,a_{r}\in\mathbb{Z}^{n}\smallsetminus\left\{
0\right\}  $ be the unique primitive inward vectors to $\gamma_{1}%
,\ldots,\gamma_{r}$\ respectively. Then
\[
\Delta_{\tau}=\left\{
{\textstyle\sum\limits_{i=1}^{r}}
\lambda_{i}a_{i};\lambda_{i}\in\mathbb{R},\text{ }\lambda_{i}\geq0\right\}
\text{ and }\Delta_{\tau}^{\circ}=\left\{
{\textstyle\sum\limits_{i=1}^{r}}
\lambda_{i}a_{i};\lambda_{i}\in\mathbb{R},\text{ }\lambda_{i}>0\right\}  .
\]

(ii) $\dim\Delta_{\tau}^{\circ}=\dim\Delta_{\tau}=n-\dim\tau$.
\end{lemma}

From the above discussion, we conclude that $\left\{  \Delta_{\tau}\right\}  $
is a fan \textit{subordinated} to $\Gamma_{\infty}$ with support
$\mathbb{R}^{n}$. We now note that if $\Delta_{\tau}\cap\mathbb{R}_{+}^{n}%
\neq\emptyset$, then $\Delta_{\tau}\cap\mathbb{R}_{+}^{n}$ is a strongly
convex polyhedral cone. We denote by \textit{Faces}($\Delta_{\tau}%
\cap\mathbb{R}_{+}^{n}$) \ the set of all the faces of cone $\Delta_{\tau}%
\cap\mathbb{R}_{+}^{n}$. Then $\cup_{\Delta_{\tau}\cap\mathbb{R}_{+}^{n}%
\neq\emptyset}$\textit{Faces}($\Delta_{\tau}\cap\mathbb{R}_{+}^{n}$) is a fan
\textit{subordinated} to $\Gamma_{\infty}$ with support $\mathbb{R}_{+}^{n}$.
Set $\Delta_{\tau}^{+}$ to be a face of $\Delta_{\tau}\cap\mathbb{R}_{+}%
^{n}\neq\emptyset$, which is also a cone, then each cone $\Delta_{\tau}^{+}%
$\ can be partitioned \ into a finite number of simplicial cones $\Delta
_{\tau,i}^{+}$. By adding new rays, each simplicial cone can be partitioned
further into a finite number of simple cones, see e.g. \cite{K-M-S}. In this
way we construct a \textit{simple fan }$\mathcal{F}$\ \textit{subordinated} to
$\Gamma_{\infty}$. From now on, we \ fix a simple fan $\mathcal{F}%
$\ subordinated to $\Gamma_{\infty}$ with support $\mathbb{R}_{+}^{n}$.

Set $\mathcal{F}_{0}$ to be the cone $\mathbb{R}_{+}^{n}$ and its faces. We
will say that $\mathcal{F}$ is \textit{trivial} if $\mathcal{F}=\mathcal{F}%
_{0}$.

Given a fan subordinated to $\Gamma_{\infty}$ with support $\mathbb{R}_{+}%
^{n}$, it is possible to obtain a \textit{conical partition of }%
$\mathbb{R}_{+}^{n}$ $\smallsetminus\left\{  0\right\}  $
(\textit{subordinated} to $\Gamma_{\infty}$) into open cones. This type of
partitions play a central role in explicit calculations of local zeta functions.

\subsection{Khovanskii Non-degeneracy Condition}

Given $a\in\mathbb{R}_{+}^{n}$, we define \textit{the face function of
}$f(x)=\sum_{m}a_{m}x^{m}$ \textit{with respect to }$a$ as
\[
f_{a}\left(  x\right)  =\sum_{m\in F\left(  a,\Gamma_{\infty}\right)  }%
a_{m}x^{m}.
\]
\ We set $T^{n}\left(  K\right)  :=\left\{  x\in K^{n};x_{1}\ldots x_{n}%
\neq0\right\}  $, for the $n$-dimensional torus considered as a $K$-analytic manifold.

\begin{definition}
\label{nondegerate}Let $f(x)=\sum_{m}a_{m}x^{m}\in K\left[  x_{1},\ldots
,x_{n},x_{1}^{-1},\ldots,x_{n}^{-1}\right]  $ be a non-constant Laurent
polynomial, and \ let $\Gamma_{\infty}$ be its Newton polytope\textit{, with
}$\dim\Gamma_{\infty}=n$\textit{. We say that }$f$\textit{\ is non-degenerate
with respect to }$a\in\mathbb{R}_{+}^{n}$, if the system of equations
\[
\left\{  f_{a}(x)=0,\nabla f_{a}\left(  x\right)  =0\right\}
\]
has no solutions in $T^{n}\left(  K\right)  $. We say that $f$ is weakly
non-degenerate \textit{with respect to }$\Gamma_{\infty}$, if \textit{ }%
$f_{a}$\textit{\ is non-degenerate with respect to any }$a\in\mathbb{R}%
_{+}^{n}$.
\end{definition}

We recall the standard non-degeneracy condition of Khovanskii.

\begin{definition}
\label{classical} Given a face $\tau$ of $\Gamma_{\infty}$, \textit{the face
function of }$f$ \textit{with respect to }$\tau$ is $f_{\tau}\left(  x\right)
:=\sum_{l\in\tau}c_{l}x^{l}$. We say that $f$\textit{\ is non-degenerate with
respect to }$\Gamma_{\infty}$, if for every face $\tau$ of $\Gamma_{\infty}$,
including $\Gamma_{\infty}$ itself, the system of equations
\[
\left\{  f_{\tau}(x)=0,\nabla f_{\tau}\left(  x\right)  =0\right\}
\]
has no solutions in $T^{n}\left(  K\right)  $.
\end{definition}

\begin{example}
\label{Example1}Take $f\left(  x,y\right)  =\left(  x^{-1}-y\right)
^{2}+x^{2}$. Then $\Gamma_{\infty}$ is a triangle with vertices at $\left(
-2,0\right)  ,\left(  0,2\right)  ,\left(  2,0\right)  $. The facet $\tau_{1}$
containing the points $\left(  -2,0\right)  $, $\left(  0,2\right)  $ has
$\left(  1,-1\right)  $ as inward vector, the facet $\tau_{2}$ containing the
points $\left(  0,2\right)  $, $\left(  2,0\right)  $ has $\left(
-1,-1\right)  $ as inward vector, and the facet $\tau_{3}$ containing the
points $\left(  -2,0\right)  $, $\left(  2,0\right)  $ has $\left(
0,1\right)  $ as inward vector. Note that the fan $\mathcal{F}$ is trivial. In
addition, $f$ is degenerate with respect to $\Gamma_{\infty}$, but $f$ is
weakly non-degenerate with respect to $\Gamma_{\infty}$.
\end{example}

\subsection{\label{torico}Toric Manifolds}

Let $A=\left\{  a_{i,j}\right\}  \in GL\left(  n,\mathbb{Z}\right)  $ with
$\det A=\pm1$. We associate to $A$ a birational morphism%
\[%
\begin{array}
[c]{cccc}%
\Psi_{A}: & \left(  K^{\times}\right)  ^{n} & \rightarrow & \left(  K^{\times
}\right)  ^{n}\\
&  &  & \\
& \left(  z_{1},\ldots,z_{n}\right)  & \rightarrow & \left(  z_{1}^{a_{1,1}%
}\cdots z_{n}^{a_{1,n}},\ldots,z_{1}^{a_{n,1}}\cdots z_{n}^{a_{n,n}}\right)  .
\end{array}
\]
Note that $\Psi_{A}$ is a group homomorphism of the algebraic group $\left(
K^{\times}\right)  ^{n}$. It is clear that $\Psi_{A}\circ\Psi_{B}=\Psi_{AB}$
and $\Psi_{A}^{-1}=\Psi_{A^{-1}}$. In addition, if there exists a subset
$J\subset\left\{  1,\ldots,n\right\}  $ such that $a_{i,j}\geq0$ for any
$i\in\left\{  1,\ldots,n\right\}  $ and $j\in J$, then $\Psi_{A}$ extends to
\[
\left\{  z\in K^{n};i\notin J\Rightarrow z_{i}\neq0\right\}  \rightarrow
K^{n}.
\]
In particular, if $a_{i,j}\geq0$ for any $i,j\in\left\{  1,\ldots,n\right\}
$, then $\Psi_{A}$ extends to a birational morphism $\Psi_{A}:K^{n}\rightarrow
K^{n}$.

Let $\mathcal{F}$ be the fixed simple fan\ subordinated to $\Gamma_{\infty}$
with support $\mathbb{R}_{+}^{n}$, and let $\mathcal{F}_{0}$ be the trivial
fan as before. Let $\Delta$ be an $n-$di\-men\-sional simple cone in
$\mathcal{F}$ and let $\tau$ be the vertex of $\Gamma_{\infty}$ such that
$F(a)=\tau$\ for \ any $a\in\Delta$. Assume that $\Delta$ is spanned \ by a
basis $a_{j}=\left(  a_{i,j}\right)  _{1\leq i\leq n}$. We set $A:=\left\{
a_{i,j}\right\}  $ and identify $A$ with $\Delta$, in particular $\Psi
_{A}:=\Psi_{\Delta}$.

We attach to $\Delta$ a copy $K_{\Delta}^{n}$ of $K^{n}$ with coordinates
$y_{\Delta}:=y=\left(  y_{1},\ldots,y_{n}\right)  $ and define the projection
morphisms%
\[%
\begin{array}
[c]{cccc}%
\sigma_{\Delta}: & K_{\Delta}^{n} & \rightarrow & K^{n}\\
&  &  & \\
& (y_{1},\ldots,y_{n}) & \rightarrow & (x_{1},\ldots,x_{n})\text{,}%
\end{array}
\]
with $x_{i}=%
{\displaystyle\prod\nolimits_{j}}
y_{j}^{a_{i,j}}$ for $i\in\left\{  1,\ldots,n\right\}  $. Thus $\sigma
_{\Delta}(y)=\Psi_{\Delta}\left(  y\right)  $. We now take $%
{\textstyle\bigsqcup\nolimits_{\substack{\Delta\in\mathcal{F}\\\dim\left(
\Delta\right)  =n}}}
K_{\Delta}^{n}$ and define an equivalence relation in this disjoint union.
Take $z_{\Delta}\in K_{\Delta}^{n}$ and $z_{\Delta^{\prime}}\in K_{\Delta
^{\prime}}^{n}$. We define $z_{\Delta}\sim z_{\Delta^{\prime}}$ if the
birational map $\Psi_{\Delta^{\prime-1}\Delta}:K_{\Delta}^{n}\rightarrow
K_{\Delta^{\prime}}^{n}$ is well defined on $z_{\Delta}\in K_{\Delta}^{n}$ and
$z_{\Delta^{\prime}}=\Psi_{\Delta^{\prime-1}\Delta}\left(  z_{\Delta}\right)
$. Then $\sim$ is an equivalence relation, cf. \cite[p. 72]{O}. Let $X\left(
\mathcal{F}\right)  $ be the quotient space $%
{\textstyle\bigsqcup\nolimits_{\substack{\Delta\in\mathcal{F}\\\dim\left(
\Delta\right)  =n}}}
K_{\Delta}^{n}/\sim$. As the gluing maps are $K$-bianalytic maps, $X\left(
\mathcal{F}\right)  $ is a $K$-analytic manifold (in the sense of Serre) with
coordinate charts $\left(  K_{\Delta}^{n},\sigma_{\Delta}\right)  $. \ We also
have a canonical projection map $\sigma:X\left(  \mathcal{F}\right)
\rightarrow K^{n}$ defined by $\sigma\mid_{K_{\Delta}^{n}}\left(  \left[
y\right]  \right)  =\sigma_{\Delta}\left(  y\right)  $ where $\left[
y\right]  $ is the equivalence class of $y\in K_{\Delta}^{n}$. This map is
proper. In \ \cite[p. 75-79]{O} this fact is proved, in the complex setting,
using sequences, this proof can be adapted to the case of $p$-adic fields.

We have a canonical embedding morphism $i_{\Delta}:\left(  K^{\times}\right)
^{n}\rightarrow\left(  K_{\Delta}^{\times}\right)  ^{n}$ defined by
$i_{\Delta}\left(  z\right)  =\Psi_{\Delta^{-1}}(z)$. This is compatible with
$\sim$ and thus we have an embedding morphism $i:\left(  K^{\times}\right)
^{n}\rightarrow X\left(  \mathcal{F}\right)  $. The image is an open dense
subset of $X\left(  \mathcal{F}\right)  $, this image is an $n$-dimensional
$K$-analytic torus.

Let $T^{n}(K)=\left\{  x\in K^{n};\text{ }x_{1}\ldots x_{n}\neq0\right\}  $ be
the $n$-dimensional $K$-analytic torus. Then, the mapping $\sigma
:\ \sigma^{-1}(T^{n}(K))\rightarrow T^{n}(K)$ is a $K$-analytic isomorphism.

It is well-known that one can define the toric manifold $X\left(
\mathcal{F}\right)  $ associated to a simple fan $\mathcal{F}$ as an algebraic
variety over $K$, and that the morphism induced by a subdivision is a proper
morphism of algebraic varieties, cf. \cite[Chapter I, Theorems 6,7,8]{K-M-S}.
Thus, since $X\left(  \mathcal{F}_{0}\right)  =K^{n}$, $\sigma:X\left(
\mathcal{F}\right)  \rightarrow K^{n}$ is a proper morphism of algebraic
varieties. By considering the $K$-analytic manifolds associated we obtain a
morphism $\sigma:X\left(  \mathcal{F}\right)  \rightarrow K^{n}$ of
$K$-analytic manifolds.

\subsubsection{\label{Nota1}Resolution of singularities}

For $a=\left(  a_{1},\ldots,a_{n}\right)  \in\mathbb{Z}^{n}\smallsetminus
\left\{  0\right\}  $, we set $\left\Vert a\right\Vert :=a_{1}+\ldots+a_{n}$.
Take $f(x)=\sum_{m\in\mathbb{Z}}c_{m}x^{m}\in K\left[  x_{1},\ldots
,x_{n},x_{1}^{-1},\ldots,x_{n}^{-1}\right]  $ a non constant Laurent
polynomial and define%
\[
f:T^{n}(K)\rightarrow K.
\]

The pair $\left(  X\left(  \mathcal{F}\right)  ,\sigma\right)  $ works as an
embedded resolution of singularities for $f$. In this section we give explicit
formulas for $f\circ\sigma$ and $\sigma^{\ast}\left(  dx_{1}\wedge\ldots\wedge
dx_{n}\right)  $ around a point of $X\left(  \mathcal{F}\right)  $.

Let $\Delta$ be an $n-$dimensional cone in $\mathcal{F}$ spanned by
$a_{1},\ldots,a_{n}$ and let $\tau$ be the vertex of $\Gamma_{\infty}$ such
that $F(a)=\tau$\ for \ any $a\in\Delta$. By the explicit description of
$\sigma_{\Delta}$ above, we have%
\begin{equation}
f_{\Delta}(y):=\left(  f\circ\sigma_{\Delta}\right)  \left(  y\right)  =%
{\displaystyle\sum\limits_{m\in\text{supp}\left(  f\right)  }}
c_{m}%
{\displaystyle\prod\limits_{j=1}^{n}}
y_{j}^{\left\langle a_{j},m\right\rangle }. \label{Ec0}%
\end{equation}
We have $\left\langle a_{j},m\right\rangle \geq d\left(  a_{j}\right)  $ by
the definition of $d\left(  a_{j}\right)  $. The equalities for all $j$ hold
if and only if the set $\left\{  m\right\}  $ coincides with the vertex $\tau
$. This implies that $f_{\Delta}(y)$ is written in the form%

\begin{equation}
\left(  f\circ\sigma_{\Delta}\right)  \left(  y\right)  =\varepsilon\left(
y\right)  \left(
{\displaystyle\prod\limits_{j=1}^{n}}
y_{j}^{d\left(  a_{j}\right)  }\right)  \text{, }\varepsilon\left(  y\right)
\in K\left[  y_{1},\ldots,y_{n}\right]  \text{, }\varepsilon\left(  0\right)
\neq0. \label{Ec1}%
\end{equation}
In particular, there exists a neighborhood $V_{0}\subset K_{\Delta}^{n}$ of
the origin such that $\left\vert \varepsilon\left(  y\right)  \right\vert
_{K}=\left\vert \varepsilon\left(  0\right)  \right\vert _{K}\neq0$ for any
$y\in V_{0}$. The above description of $\sigma_{\Delta}$ also implies%
\begin{equation}
\sigma_{\Delta}^{\ast}\left(  dx_{1}\wedge\ldots\wedge dx_{n}\right)  =\left(
\pm1\right)  \left(
{\displaystyle\prod\limits_{j=1}^{n}}
y_{j}^{\left\Vert a_{j}\right\Vert -1}\right)  dy_{1}\wedge\ldots\wedge
dy_{n}, \label{Ec3}%
\end{equation}
for any $y\in V_{0}$.

Let $b\neq0$ be a point of $K_{\Delta}^{n}\smallsetminus\left(  K_{\Delta
}^{\times}\right)  ^{n}$. By renaming the coordinates, we assume $b=\left(
0,\dots,0,b_{r+1},\dots,b_{n}\right)  $ with $b_{i}\in K^{\times}$ for
$r+1\leq i\leq n$. Let $\Delta^{\prime}$ be the face of $\Delta$ spanned by
$a_{1},\ldots,a_{r}$ and let $\tau^{\prime}$ be the face of $\Gamma_{\infty}$
such that $F(a)=\tau^{\prime}$ for all $a\in\Delta^{\prime}$. Then, for $m\in
$supp$\left(  f\right)  $, $\left\langle a_{j},m\right\rangle =d\left(
a_{j}\right)  $\ holds for all $j\in\left\{  1,\ldots,r\right\}  $\ if and
only if $m\in\tau^{\prime}$. Hence we may write (\ref{Ec0}) as%
\[
\left(  f\circ\sigma_{\Delta}\right)  \left(  y\right)  =\left(
{\displaystyle\prod\limits_{j=1}^{r}}
y_{j}^{d\left(  a_{j}\right)  }\right)  \left(  f_{\Delta,\tau^{\prime}%
}\left(  y\right)  +h_{\Delta,\tau^{\prime}}(y)\right)  ,
\]
with%
\begin{align*}
f_{\Delta,\tau^{\prime}}\left(  y\right)   &  =%
{\displaystyle\sum\limits_{m\notin\tau^{\prime}\cap\text{supp}\left(
f\right)  }}
c_{m}%
{\displaystyle\prod\limits_{j=r+1}^{n}}
y_{j}^{\left\langle a_{j},m\right\rangle }\in K\left[  y_{r+1},y_{r+1}%
^{-1},\ldots,y_{n},y_{n}^{-1}\right]  \text{, }\\
h_{\Delta,\tau^{\prime}}(y)  &  \in%
{\displaystyle\sum\limits_{j=1}^{r}}
y_{j}K\left[  y_{1},\ldots,y_{r},y_{r+1},y_{r+1}^{-1},\ldots,y_{n},y_{n}%
^{-1}\right]  .
\end{align*}
Note that $h_{\Delta,\tau^{\prime}}\left(  b\right)  =0$. Two cases happen:
(i) $f_{\Delta,\tau^{\prime}}\left(  b\right)  \neq0$, (ii) $f_{\Delta
,\tau^{\prime}}\left(  b\right)  =0$. In the first case,
\begin{equation}
\left(  f\circ\sigma_{\Delta}\right)  \left(  y\right)  =\varepsilon\left(
y\right)  \left(
{\displaystyle\prod\limits_{j=1}^{r}}
y_{j}^{d\left(  a_{j}\right)  }\right)  \text{, }\varepsilon\left(  y\right)
\in K\left[  y_{1},\ldots,y_{r},y_{r+1},y_{r+1}^{-1},\ldots,y_{n},y_{n}%
^{-1}\right]  \text{, } \label{Ec4AA}%
\end{equation}
with $\varepsilon\left(  b\right)  \neq0$, and%
\begin{align}
\sigma_{\Delta}^{\ast}\left(  dx_{1}\wedge\ldots\wedge dx_{n}\right)   &
=\eta\left(  y\right)  \left(
{\displaystyle\prod\limits_{j=1}^{r}}
y_{j}^{\left\Vert a_{j}\right\Vert -1}\right)  dy_{1}\wedge\ldots\wedge
dy_{n}\text{, }\label{EC4AB}\\
\eta\left(  y\right)   &  \in K\left[  y_{1},\ldots,y_{r},y_{r+1},y_{r+1}%
^{-1},\ldots,y_{n},y_{n}^{-1}\right]  \text{, }\eta\left(  b\right)
\neq0.\nonumber
\end{align}
In particular, there exists an open neighborhood $V_{b}\subset K_{\Delta}^{n}$
of $b$ such that $\left\vert \varepsilon\left(  y\right)  \right\vert
_{K}=\left\vert \varepsilon\left(  b\right)  \right\vert _{K}$ and $\left\vert
\eta\left(  y\right)  \right\vert _{K}=\left\vert \eta\left(  b\right)
\right\vert _{K}$ for $y\in V_{b}$.

Suppose that $f_{\Delta,\tau^{\prime}}\left(  b\right)  =0$. We claim that
there exists $l\in\left\{  r+1,\ldots,n\right\}  $ such that $\frac{\partial
f_{\Delta,\tau^{\prime}}}{\partial y_{l}}\left(  b\right)  \neq0$. Choose
$b_{i}\in K^{\times}$, for $1\leq i\leq r$ and set\
\[
\widetilde{b}=\left(  b_{1},\dots,b_{r},b_{r+1},\dots,b_{n}\right)  \in\left(
K^{\times}\right)  ^{n}.
\]
Put $f_{\tau^{\prime}}\left(  x\right)  =\sum_{m\in\tau^{\prime}}c_{m}x^{m}$.
Then%
\begin{equation}
f_{\tau^{\prime}}\circ\sigma_{\Delta}\left(  y\right)  =f_{\Delta,\tau
^{\prime}}\left(  y\right)
{\displaystyle\prod\limits_{j=1}^{r}}
y_{j}^{d\left(  a_{j}\right)  }. \label{Ec5}%
\end{equation}
Hence $f_{\tau^{\prime}}\circ\sigma_{\Delta}\left(  \widetilde{b}\right)  =0$.
Since $\sigma_{\Delta}:\left(  K^{\times}\right)  ^{n}\rightarrow T^{n}(K)$ is
an isomorphism of $K$-analytic manifolds, the non-degeneracy of $f$ implies
$\nabla\left(  f_{\tau^{\prime}}\circ\sigma_{\Delta}\right)  \left(
\widetilde{b}\right)  \neq0$. Since $f_{\Delta,\tau^{\prime}}\left(  y\right)
\in K\left[  y_{r+1},y_{r+1}^{-1},\ldots,y_{n},y_{n}^{-1}\right]  $, we have
$f_{\Delta,\tau^{\prime}}\left(  \widetilde{b}\right)  =f_{\Delta,\tau
^{\prime}}\left(  b\right)  =0$ and $\frac{\partial f_{\Delta,\tau^{\prime}}%
}{\partial y_{l}}\left(  \widetilde{b}\right)  =\frac{\partial f_{\Delta
,\tau^{\prime}}}{\partial y_{l}}\left(  b\right)  $ for $r+1\leq l\leq n$. By
(\ref{Ec5}), we obtain
\[
\frac{\partial\left(  f_{\tau^{\prime}}\circ\sigma_{\Delta}\right)  }{\partial
y_{l}}\left(  \widetilde{b}\right)  =\left\{
\begin{array}
[c]{lll}%
0 & \text{if} & 1\leq l\leq r\\
&  & \\
\frac{\partial f_{\Delta,\tau^{\prime}}}{\partial y_{l}}\left(  b\right)
{\displaystyle\prod\limits_{j=1}^{r}}
b_{j}^{d\left(  a_{j}\right)  } & \text{if} & r+1\leq l\leq n.
\end{array}
\right.
\]
This implies the desired claim. By renaming the coordinates if necessary, we
assume $\frac{\partial f_{\Delta,\tau^{\prime}}}{\partial y_{r+1}}\left(
b\right)  \neq0$. Since $\frac{\partial h_{\Delta,\tau^{\prime}}}{\partial
y_{r+1}}\left(  b\right)  =0$, letting $y_{r+1}^{\prime}=f_{\Delta
,\tau^{\prime}}\left(  y\right)  +h_{\Delta,\tau^{\prime}}(y)$ and
$y_{j}^{\prime}=y_{j}$ for $j\neq r+1$, we see that $y^{\prime}=(y_{1}%
,\dots,y_{r},y_{r+1}^{\prime},y_{r+2},\dots,y_{n})$ becomes a coordinate
system in a neighborhood $V_{b}$ of $b$ and obtain%

\begin{equation}
\left(  f\circ\sigma_{\Delta}\right)  \left(  y^{\prime}\right)  =\left(
{\displaystyle\prod\limits_{j=1}^{r}}
y_{j}^{\prime d\left(  a_{j}\right)  }\right)  y_{r+1}^{\prime}\text{.}
\label{Ec5A}%
\end{equation}
From (\ref{Ec3}), we also obtain%
\begin{align}
&  \sigma_{\Delta}^{\ast}\left(  dx_{1}\wedge\ldots\wedge dx_{n}\right)
=\nonumber\\
&  \eta\left(  y^{\prime}\right)  \left(
{\displaystyle\prod\limits_{j=1}^{r}}
y_{j}^{\prime\left\Vert a_{j}\right\Vert -1}\right)  dy_{1}^{\prime}%
\wedge\ldots\wedge dy_{n}^{\prime}, \label{Ec5C}%
\end{align}
where $\eta\left(  y^{\prime}\right)  $ is a $K$-analytic function on $V_{b}$
such that $\left\vert \eta\left(  b^{\prime}\right)  \right\vert _{K}\neq0$,
where $b^{\prime}$ denotes the coordinates of $b$ with respect to the new
coordinate system $y^{\prime}$. There exists an open neighborhood
$V_{b}^{\prime}\subset V_{b}$ of $b^{\prime}$ such that $\left\vert
\eta\left(  y^{\prime}\right)  \right\vert _{K}=\left\vert \eta\left(
b^{\prime}\right)  \right\vert _{K}\neq0$ for any $y^{\prime}\in V_{b}%
^{\prime}$.

Finally, suppose that $b\in\left(  K_{\Delta}^{\times}\right)  ^{n}$, which
implies $\sigma_{\Delta}(b)\neq0$. If $f\left(  \sigma_{\Delta}(b)\right)
=0$, by using the weak non-degeneracy of $f$ with respect to $\Gamma_{\infty}%
$\ there exists $i\in\left\{  1,2,\ldots,n\right\}  $ such that $\frac
{\partial f}{\partial x_{i}}\left(  \sigma_{\Delta}(b)\right)  \neq0$. Now,
since $\sigma_{\Delta}|_{T^{n}(K)}$ is a $K-$analytic isomorphism, we may
define a new coordinate system $y^{\prime}=(y_{1}^{\prime},\ldots
,y_{n}^{\prime})$ on a neighborhood $V_{b}$ of $b$ as follows:%
\[
(y_{1}^{\prime},\ldots,y_{n}^{\prime})=\left(  f\circ\sigma_{\Delta}%
,x_{1}\circ\sigma_{\Delta},\ldots,x_{i-1}\circ\sigma_{\Delta},x_{i+1}%
\circ\sigma_{\Delta},\ldots,x_{n}\circ\sigma_{\Delta},\right)  .
\]
With this new coordinate system we have
\[
\sigma_{\Delta}^{\ast}\left(  dx_{1}\wedge\ldots\wedge dx_{n}\right)  =\left(
-1\right)  ^{i-1}\left[  \frac{\partial f}{\partial x_{i}}\left(
\sigma_{\Delta}(b)\right)  \right]  ^{-1}dy_{1}^{\prime}\wedge\ldots\wedge
dy_{n}^{\prime}.
\]
Therefore
\begin{equation}
(f\circ\sigma_{\Delta})(y)=y_{1}^{\prime}, \label{Ec5D}%
\end{equation}%
\begin{equation}
\sigma_{\Delta}^{\ast}\left(  dx_{1}\wedge\ldots\wedge dx_{n}\right)
=\eta\left(  y^{\prime}\right)  dy_{1}^{\prime}\wedge\ldots\wedge
dy_{n}^{\prime}, \label{Ec5E}%
\end{equation}
with $\eta\left(  y^{\prime}\right)  $ a $K$-analytic function defined on
$V_{b}$ such that $\left\vert \eta\left(  b\right)  \right\vert _{K}\neq0$ and
$\left\vert \eta\left(  y^{\prime}\right)  \right\vert _{K}=\left\vert
\eta\left(  b^{\prime}\right)  \right\vert _{K}$for any $y\in V_{b}$.

If $f\left(  \sigma_{\Delta}(b)\right)  \neq0$, we define a new coordinate
system $y^{\prime}=(y_{1}^{\prime},\ldots,y_{n}^{\prime})$ by $y_{i}^{\prime
}=x_{i}\circ\sigma_{\Delta}$. Then there exists a neighborhood $V_{b}$ of $b$
such that $\left\vert (f\circ\sigma_{\Delta})(y)\right\vert _{K}=\left\vert
(f\circ\sigma_{\Delta})(b)\right\vert _{K}$ and $\sigma_{\Delta}^{\ast}\left(
dx_{1}\wedge\ldots\wedge dx_{n}\right)  =dy_{1}\wedge\ldots\wedge dy_{n}$ for
any $y\in V_{b}$.

\subsection{\label{Nota2}A hypothesis on the critical locus of $f$}

We consider $f$ as a regular function on $T^{n}(K)$. The critical set of $f$
is $C_{f}:=C_{f}\left(  K\right)  =\left\{  x\in T^{n}(K)\ ;\ \nabla f\left(
x\right)  =0\right\} $. Notice that by the non-degeneracy condition on $f$,
$C_{f}\cap f^{-1}(0)=\emptyset$. Later on we will use the following condition:
(A) $C_{f}=\emptyset$; (B) let $\mathcal{F}$ be a fixed simple, non trivial,
fan \ subordinated to $\Gamma_{\infty}$. For any $n-$dimensional cone $\Delta$
in $\mathcal{F}$ spanned by $a_{1},\ldots,a_{n}$, $d(a_{j})\neq0$ \ for any
$j$ in (\ref{Ec1}). We will call these conditions \textit{Hypothesis H1}.

Hypothesis H1 is necessary to assure the vanishing of the twisted local zeta
functions, and thus, to use Igusa's method for estimating $p$-adic oscillatory integrals.

Let $b\in X\left(  \mathcal{F}\right)  $ and $a=\sigma\left(  b\right)  $. If
$f\left(  a\right)  \neq0$, by hypothesis H1, there is a local coordinate
system of the form $y^{\prime}=\left(  f\left(  a\right)  ^{-1}f\left(
x\right)  -1,y_{2},\ldots,y_{n}\right)  $ in a neighborhood $V_{b}$ of
$b$,\ then%
\begin{align}
\left(  f\circ\sigma\right)  \left(  y\right)   &  =f\left(  a\right)  \left(
1+y_{1}^{\prime}\right)  ,\nonumber\\
\sigma^{\ast}\left(  dx_{1}\wedge\ldots\wedge dx_{n}\right)   &  =\eta\left(
y^{\prime}\right)  dy_{1}^{\prime}\wedge\ldots\wedge dy_{n}^{\prime},
\label{Ec6}%
\end{align}
and $\left\vert \eta\left(  y^{\prime}\right)  \right\vert _{K}=\left\vert
\eta\left(  b^{\prime}\right)  \right\vert _{K}$ for any $y\in V_{b}.$

\section{\label{zeta_Sect} Local Zeta Functions}

In this section we attach to a Laurent polynomial in $n$ variables a local
zeta function and show that it has a meromorphic continuation to the whole
complex plane. We also give some results about the poles of the meromorphic continuation.

\subsection{Quasicharacters}

Let $K$ be a $p-$adic field, i.e. $[K:\mathbb{Q}_{p}]<\infty$, where
$\mathbb{Q}_{p}$ denotes the field of $p$-adic numbers. Let $R_{K}$\ be the
valuation ring of $K$, $P_{K}$ the maximal ideal of $R_{K}$, and $\overline
{K}=R_{K}/P_{K}$ \ the residue field of $K$. The cardinality of the residue
field of $K$ is denoted by $q$, thus $\overline{K}=\mathbb{F}_{q}$. For $z\in
K$, $ord\left(  z\right)  \in\mathbb{Z}\cup\{+\infty\}$ \ denotes the
valuation of $z$, and $\left\vert z\right\vert _{K}=q^{-ord\left(  z\right)
}$, $ac$ $z=z\mathfrak{p}^{-ord(z)}$, where $\mathfrak{p}$ is a fixed
uniformizing parameter of $R_{K}$.

We equip $K^{n}$ with the norm $\left\Vert \left(  x_{1},\ldots,x_{n}\right)
\right\Vert _{K}:=\max\left(  \left\vert x_{1}\right\vert _{K},\ldots
,\left\vert x_{n}\right\vert _{K}\right)  $. Then $\left(  K^{n},\left\Vert
\cdot\right\Vert _{K}\right)  $ is a complete metric space and the metric
topology is equal to the product topology.

Let $\omega$ be a quasicharacter of $K^{\times}$, i.e. a continuous
homomorphism from $K^{\times}$ into $\mathbb{C}^{\times}$. The set of
quasicharacters form an Abelian group denoted as $\Omega\left(  K^{\times
}\right)  $. We define an element $\omega_{s}$ of $\Omega\left(  K^{\times
}\right)  $ for every $s\in\mathbb{C}$ as $\omega_{s}\left(  x\right)
=\left\vert x\right\vert _{K}^{s}=q^{-sord\left(  x\right)  }$. If, for every
$\omega$ in $\Omega\left(  K^{\times}\right)  $, we choose $s\in\mathbb{C}%
$\ satisfying $\omega\left(  \mathfrak{p}\right)  =q^{-s}$, then
$\omega\left(  x\right)  =\omega_{s}\left(  x\right)  \chi\left(  ac\text{
}x\right)  $ in which $\chi:=\omega\mid_{R_{K}^{\times}}$. We denote the
conductor of $\chi$ as $c\left(  \chi\right)  $. Hence $\Omega\left(
K^{\times}\right)  $ is isomorphic to $\mathbb{C}/\left(  2\pi\sqrt{-1}/\ln
q\right)  \mathbb{\times}\left(  R_{K}^{\times}\right)  ^{\ast}$, where
$\left(  R_{K}^{\times}\right)  ^{\ast}$ is the group of characters of
$R_{K}^{\times}$, and $\Omega\left(  K^{\times}\right)  $ is a one dimensional
complex manifold. We note that $\sigma\left(  \omega\right)
:=\operatorname{Re}(s)$ depends only on $\omega$, and $\left\vert
\omega\left(  x\right)  \right\vert =\omega_{\sigma\left(  \omega\right)
}\left(  x\right)  $. Given an interval $(a,b)$, we define an open subset of
$\Omega\left(  K^{\times}\right)  $ by
\[
\Omega_{(a,b)}\left(  K^{\times}\right)  =\left\{  \omega\in\Omega\left(
K^{\times}\right)  ;\sigma\left(  \omega\right)  \in(a,b)\right\}  .
\]
For further details we refer the reader to \cite{I2}.

\subsection{Meromorphic Continuation of Local Zeta Functions}

The following result will be used later frequently.

\begin{lemma}
\label{lemma3} Take $a\in K$, $\omega\in\Omega\left(  K^{\times}\right)  $ and
$N\in\mathbb{Z\smallsetminus}\left\{  0\right\}  $. Take also $n,e\in$
$\mathbb{N}$, with $n>0$, and put $\chi=\omega\mid_{R_{K}^{\times}}$. Then%
\[%
{\displaystyle\int\limits_{a+\mathfrak{p}^{e}R_{K}\smallsetminus\left\{
0\right\}  }}
\omega\left(  z\right)  ^{N}\left\vert z\right\vert _{K}^{n-1}\left\vert
dz\right\vert =\left\{
\begin{array}
[c]{llll}%
\left(  1-q^{-1}\right)  \frac{q^{-en-eNs}}{1-q^{-n-Ns}} &  & \text{if} &
\begin{array}
[c]{c}%
a\in\mathfrak{p}^{e}R_{K}\\
\chi^{N}=1
\end{array}
\\
&  &  & \\
q^{-e}\omega\left(  a\right)  ^{N}\left\vert a\right\vert _{K}^{n-1} &  &
\text{if} &
\begin{array}
[c]{c}%
a\notin\mathfrak{p}^{e}R_{K}\\
\chi^{N}\mid_{U^{\prime}}=1
\end{array}
\\
&  &  & \\
0 &  &  & \text{all other cases,}%
\end{array}
\right.
\]
in which $U^{\prime}=1+\mathfrak{p}^{e}a^{-1}R_{K}$. In addition, the integral
converges on $\operatorname{Re}(s)>\frac{-n}{N}$, if $N>0$, and on
$\operatorname{Re}(s)<\frac{n}{\left\vert N\right\vert }$, if $N<0$. In
addition, if $N=0$ the above integral converges to a non-zero value.
\end{lemma}

\begin{proof}
The proof of the lemma is an easy variation of the one given for Lemma 8.2.1
in \cite{I2}.
\end{proof}

\bigskip

We recall that a locally constant function on $K^{n}$ with compact support is
called a \textit{Bruhat-Schwartz function}, these functions form a
$\mathbb{C}$-vector space denoted as $S(K^{n})$.

For $a=\left(  a_{1},\ldots,a_{n}\right)  \in\mathbb{Z}^{n}\smallsetminus
\left\{  0\right\}  $, set $\left\Vert a\right\Vert =a_{1}+\ldots+a_{n}$ as
before, and%
\[
\mathcal{P}\left(  a\right)  :=\left\{
\begin{array}
[c]{lll}%
\left\{  -\frac{\left\Vert a\right\Vert }{d\left(  a\right)  }+\frac{2\pi
\sqrt{-1}\mathbb{Z}}{d\left(  a\right)  \ln q}\right\}  & \text{if} & d\left(
a\right)  \neq0\\
&  & \\
\varnothing & \text{if} & d\left(  a\right)  =0.
\end{array}
\right.
\]

Let $\mathcal{F}$ be the fixed simple fan subordinated to $\Gamma_{\infty}$ as
before. Denote by $edges(\mathcal{F})$, the set of all edges of the cones in
$\mathcal{F}$ as before. Set%
\[
A\left(  \mathcal{F}\right)  :=%
{\displaystyle\bigcup\limits_{\substack{a\in edges(\mathcal{F})\\d\left(
a\right)  \neq0}}}
\left\{  \frac{\left\Vert a\right\Vert }{-d\left(  a\right)  };\text{
}d\left(  a\right)  <0\right\}  ,
\]%
\[
B\left(  \mathcal{F}\right)  :=%
{\displaystyle\bigcup\limits_{\substack{a\in edges(\mathcal{F})\\d\left(
a\right)  \neq0}}}
\left\{  \frac{\left\Vert a\right\Vert }{-d\left(  a\right)  };\text{
}d\left(  a\right)  >0\right\}  ,
\]%
\[
\alpha:=\alpha\left(  \mathcal{F}\right)  =\left\{
\begin{array}
[c]{lll}%
\min_{\gamma\in A\left(  \mathcal{F}\right)  }\gamma\text{,} & \text{if} &
A\left(  \mathcal{F}\right)  \neq\emptyset\\
&  & \\
+\infty\text{,} & \text{if} & A\left(  \mathcal{F}\right)  =\emptyset,
\end{array}
\right.
\]
and
\[
\beta:=\beta\left(  \mathcal{F}\right)  =\max_{\gamma\in B\left(
\mathcal{F}\right)  \cup\left\{  -1\right\}  }\gamma.
\]

\begin{definition}
\label{zeta_function}Given $f$ a Laurent polynomial, $\Phi$ a Bruhat-Schwartz
function, and $\omega\in\Omega\left(  K^{\times}\right)  $, we attach to these
data the following local zeta function:%
\[
Z_{\Phi}\left(  \omega,f\right)  =Z_{\Phi}\left(  s,\chi,f\right)  =%
{\displaystyle\int\limits_{T^{n}\left(  K\right)  }}
\Phi\left(  x\right)  \omega\left(  f\left(  x\right)  \right)  \left\vert
dx\right\vert ,
\]
where $\left\vert dx\right\vert $ is the normalized Haar measure of $K^{n}$,
which is the measure induced by an $n$-degree differential form $dx$.
\end{definition}

\begin{theorem}
\label{Th1}Let $f$ be weakly non- degenerate Laurent polynomial with respect
to $\Gamma_{\infty}$, and let $\mathcal{F}$ be a fixed simple, non trivial,
fan \ subordinated to $\Gamma_{\infty}$. Then the following assertions hold:

\noindent(i) $Z_{\Phi}\left(  \omega,f\right)  $\ converges for $\omega\in$
$\Omega_{\left(  \beta,\alpha\right)  }\left(  K^{\times}\right)  $.

\noindent(ii) $Z_{\Phi}\left(  \omega,f\right)  $\ has a meromorphic
continuation to $\Omega\left(  K^{\times}\right)  $ as a rational function of
$\omega\left(  q\right)  $, and the poles belong to
\[%
{\displaystyle\bigcup\limits_{a\in edges(\mathcal{F})}}
\mathcal{P}(a)\cup\left\{  -1+\frac{2\pi\sqrt{-1}\mathbb{Z}}{\ln q}\right\}
\text{.}%
\]
In addition, the multiplicity of any pole is $\leq n$.
\end{theorem}

\begin{proof}
We pick a pair $\left(  X\left(  \mathcal{F}\right)  ,\sigma\right)  $ as in
Section \ref{torico} and use all the notation introduced there. By using the
fact that $\sigma:\sigma^{-1}\left(  T^{n}\left(  K\right)  \right)
\rightarrow T^{n}\left(  K\right)  $ is a $K$-analytic isomorphism, we have
\[
Z_{\Phi}(\omega,f)=%
{\displaystyle\int\limits_{T^{n}\left(  K\right)  }}
\Phi\left(  x\right)  \omega\left(  f\left(  x\right)  \right)  \left\vert
dx\right\vert =%
{\displaystyle\int\limits_{\sigma^{-1}\left(  T^{n}\left(  K\right)  \right)
}}
\Phi\circ\sigma\left(  y\right)  \omega\left(  f\circ\sigma\left(  y\right)
\right)  \left\vert \sigma^{\ast}\left(  dx\right)  \right\vert .
\]
Since $\sigma$ is a proper map \ and $S=$supp$\left(  \Phi\right)  $ is
compact open, we see that $\sigma^{-1}\left(  S\right)  $ is a compact subset
of $X\left(  \mathcal{F}\right)  $. For every point $b\in\sigma^{-1}\left(
S\right)  $ \ there exists a neighborhood $V_{b}$ such that (\ref{Ec0}%
)-(\ref{Ec5E}) \ hold, and by the compactness of $\sigma^{-1}\left(  S\right)
$, there is a finite covering of $\sigma^{-1}\left(  S\right)  $, say $U_{i}$,
$i=1,2,\ldots,M$, where all these formulas hold. Now by taking $U_{1}$,
$U_{2}\smallsetminus U_{1},\ldots,U_{k}\smallsetminus\cup_{i=1}^{k-1}U_{i}$,
etc., we may assume that the $U_{i}$ are already disjoint and non-empty. After
embedding each of these subsets in $K^{n}$ and decomposing them into cosets
modulo $P_{K}^{e}$, where $e$ is a fixed natural number, we get a disjoint
open covering $V_{i}$, $i=1,2,\ldots,M^{\prime}$ of $S\cap T^{n}\left(
K\right)  $ such that each $V_{i}=c_{i}+\left(  P_{K}^{e}\right)  ^{n}$,
$c_{i}\in K^{n}$ for $i=1,2,\ldots,M^{\prime}$. In addition, we choose the
open sets $V_{i}$'s in such way that $\omega\left(  \varepsilon\left(
y\right)  \right)  $ is constant on $V_{i}$.

Therefore $Z_{\Phi}(\omega,f)$ becomes a finite sum of integrals of the
following types: First, if (\ref{Ec1})-(\ref{Ec3}) or (\ref{Ec4AA}%
)-(\ref{EC4AB}) hold, then
\begin{equation}
J_{0}(\omega)=q^{-e\left(  n-r\right)  }\Phi\left(  \sigma\left(  b\right)
\right)  \omega\left(  \varepsilon\left(  b\right)  \right)  \left\vert
\eta\left(  b\right)  \right\vert _{K}%
{\displaystyle\prod\limits_{j=1}^{r}}
\text{ }%
{\displaystyle\int\limits_{c_{j}+\mathfrak{p}^{e}R_{K}\smallsetminus\left\{
0\right\}  }}
\omega\left(  y_{j}\right)  ^{d\left(  a_{j}\right)  }\left\vert
y_{j}\right\vert _{K}^{\left\Vert a_{j}\right\Vert -1}\left\vert
dy_{j}\right\vert , \label{JAA}%
\end{equation}
where $b$ is point in $X\left(  \mathcal{F}\right)  $, $c=\left(  c_{1}%
,\ldots,c_{n}\right)  \in K^{n}$, $e\in\mathbb{N}$, and $1\leq r\leq n$. We
include (\ref{Ec1})-(\ref{Ec3}) and (\ref{Ec4AA})-(\ref{EC4AB}) in the same
case by allowing $r=n$; Second, if (\ref{Ec5A})-(\ref{Ec5C}) hold, then
\begin{equation}
J_{1}(\omega)=q^{-e\left(  n-r-1\right)  }\Phi\left(  \sigma\left(  b^{\prime
}\right)  \right)  \left\vert \eta\left(  b^{\prime}\right)  \right\vert
_{K}\times\label{JB}%
\end{equation}%
\[
\left(
{\displaystyle\prod\limits_{j=1}^{r}}
\text{ }%
{\displaystyle\int\limits_{c_{j}+\mathfrak{p}^{e}R_{K}\smallsetminus\left\{
0\right\}  }}
\omega\left(  y_{j}^{\prime}\right)  ^{d\left(  a_{j}\right)  }\left\vert
y_{j}^{\prime}\right\vert _{K}^{\left\Vert a_{j}\right\Vert -1}\left\vert
dy_{j}^{\prime}\right\vert \right)  \left(  \text{ }%
{\displaystyle\int\limits_{c_{r+1}+\mathfrak{p}^{e}R_{K}\smallsetminus\left\{
0\right\}  }}
\omega\left(  y_{r+1}^{\prime}\right)  \left\vert dy_{r+1}^{\prime}\right\vert
\right)  ,
\]
where $1\leq r\leq n-1$; Third, if (\ref{Ec5D})-(\ref{Ec5E}) hold, then
\begin{equation}
J_{2}(\omega)=q^{-\left(  n-1\right)  e}\Phi\left(  \sigma\left(  b^{\prime
}\right)  \right)  \left\vert \eta\left(  b^{\prime}\right)  \right\vert _{K}%
{\displaystyle\int\limits_{c_{1}+\mathfrak{p}^{e}R_{K}}}
\omega\left(  y_{1}^{\prime}\right)  \left\vert dy_{1}^{\prime}\right\vert .
\label{JC}%
\end{equation}

Finally, we note if $f\circ\sigma\left(  b\right)  \neq0$, then by the
discussion at the last paragraph of Section \ref{Nota1}, the corresponding
integral is a holomorphic function of $s$.

The parts (i)-(ii) follow by applying Lemma \ref{lemma3} to integrals
(\ref{JAA})-(\ref{JC}).
\end{proof}

\begin{remark}
Let $f\left(  x\right)  =\sum c_{m}x^{m}$ be a weakly non-degenerate Laurent
polynomial \ with coefficients in $R_{K}^{\times}$. Assume that $\mathcal{F}%
=\mathcal{F}_{0}$, and that $\Phi$ is the characteristic function of
$R_{K}^{n}$, and $\omega=\omega_{s}$. Then the first meet locus of \ any
integer vector in $\mathbb{R}_{+}^{n}$ is a point, say $m_{0}=\left(
m_{0,1},\ldots,m_{0,n}\right)  $. In addition,%
\begin{align*}
Z_{\Phi}(\omega_{s},f)  &  =%
{\displaystyle\sum\limits_{\left(  a_{1},\ldots,a_{n}\right)  \in
\mathbb{N}^{n}}}
\text{ }%
{\displaystyle\int\limits_{\mathfrak{p}^{a_{1}}R_{K}^{\times}\times
\cdots\times\mathfrak{p}^{a_{n}}R_{K}^{\times}}}
\left\vert f\left(  x\right)  \right\vert _{K}^{s}\left\vert dx\right\vert \\
&  =\left(  1-q^{-1}\right)  ^{n}%
{\displaystyle\sum\limits_{\left(  a_{1},\ldots,a_{n}\right)  \in
\mathbb{N}^{n}}}
\text{ }q^{-\left\Vert a\right\Vert -\left\langle a,m_{0}\right\rangle s}=%
{\displaystyle\prod\limits_{i=1}^{n}}
\left(  \frac{1-q^{-1}}{1-q^{-1-m_{0,i}s}}\right)  .
\end{align*}

It is not difficult to show that in general case, we have%
\[
Z_{\Phi}(\omega,f)=\frac{L(q^{-s})}{%
{\displaystyle\prod\limits_{i=1}^{n}}
\left(  1-q^{-1-m_{0,i}s}\right)  },
\]
where $L(q^{-s})$ is a polynomial in $q^{-s}$ with rational coefficients.
\end{remark}

\begin{remark}
\label{Nota2AA}(i) Take $f(x,y)=x^{-2}-2x^{-1}y+y^{2}+x^{2}$, as in Example
\ref{Example1}, then $\mathcal{F}$ is the trivial fan. Take $\Phi$ the
characteristic function of $\left(  \mathfrak{p}R_{K}\right)  ^{2}$, and
$\omega=\omega_{s}$. Then
\begin{align*}
Z_{\Phi}(\omega_{s},f)  &  ={\int\limits_{\left(  \mathfrak{p}R_{K}%
\smallsetminus\left\{  0\right\}  \right)  ^{2}}}\left\vert f\left(
x,y\right)  \right\vert _{K}^{s}\left\vert dxdy\right\vert \\
&  =\sum\limits_{a=1}^{\infty}\sum\limits_{b=1}^{\infty}{\int
\limits_{\mathfrak{p}^{a}R_{K}^{\times}\times\mathfrak{p}^{b}R_{K}^{\times}}%
}\left\vert f\left(  x,y\right)  \right\vert _{K}^{s}\left\vert
dxdy\right\vert =\frac{\left(  1-q^{-1}\right)  q^{-2+2s}}{1-q^{-1+2s}}.
\end{align*}
Note that the integral converges for $\operatorname{Re}\left(  s\right)
<\frac{1}{2}$. Thus local zeta functions $Z_{\Phi}(\omega,f)$ may have poles
with positive real parts.

(ii) In dimension $2$, an explicit formula for $Z_{\Phi}(\omega_{s},f)$, when
$\Phi$ is the characteristic function of $\left(  R_{K}\smallsetminus\left\{
0\right\}  \right)  ^{2}$, similar to the one given in \cite{D-H}\ holds. Let
$\mathcal{L}$ be a simple conical \ partition of $\mathbb{R}^{2}%
\smallsetminus\left\{  0\right\}  $ subordinated to $\Gamma_{\infty}$. Then,
the intersection of $\mathcal{L}$ with the first quadrant of $\mathbb{R}^{2}$
gives a simple conical partition of the first quadrant, in addition, the
corresponding skeleton is the union the vectors in $edges\left(
\mathcal{L}\right)  $ contained in the first quadrant and the vectors of a
canonical basis of $\mathbb{R}^{2}$. This simple construction does not work in
dimensions greater than two.
\end{remark}

From now on, we will assume that $\mathcal{F}$ is not trivial.

\subsection{Some Additional Remarks on Poles of Local Zeta Functions}

As a consequence of Theorem \ref{Th1}, the mapping $\Phi\rightarrow Z_{\Phi
}\left(  \omega,f\right)  $ defines a meromorphic distribution on $S(K^{n})$.
Denote this functional by $Z_{\bullet}\left(  \omega\right)  $. The set of
poles of $Z_{\bullet}\left(  \omega\right)  $ is the set of poles of all the
meromorphic functions $Z_{\Phi}\left(  \omega,f\right)  $\ when $\Phi$ runs
through $S(K^{n})$.

\begin{lemma}
\label{prop2}Assume that $A\left(  \mathcal{F}\right)  \neq\emptyset$. Given
$l\in\mathbb{N}$, with $1\leq l\leq n$, define
\[
\mathcal{L}_{l}(\alpha)=\left\{
\begin{array}
[c]{c}%
\Delta\in\mathcal{F};\Delta\text{ has exactly }l\text{ edges, }a_{k}\text{,}\\
\text{satisfying }\frac{\left\Vert a_{k}\right\Vert }{-d\left(  a_{k}\right)
}=\alpha\text{ for }k=1,\ldots,l.
\end{array}
\right\}
\]
If $\max_{l}\left\{  \mathcal{L}_{l}(\alpha)\neq\emptyset\right\}  =n$, then
$Z_{\bullet}\left(  \omega\right)  $ has a pole $s$ with multiplicity $n$
satisfying $\operatorname{Re}(s)=\alpha$.
\end{lemma}

\begin{proof}
We use all the notation introduced in Paragraph \ref{Nota1}. Pick $\Phi>0$
(later we will impose more restrictions on $\Phi$) and $\omega=\omega_{s}$. To
prove the result, it is sufficient to show that%
\begin{equation}
\lim_{s\rightarrow\alpha}\left(  1-q^{s-\alpha}\right)  ^{n}Z_{\Phi}\left(
\omega_{s},f\right)  >0. \label{cond}%
\end{equation}

Since $Z_{\Phi}\left(  \omega,f\right)  $ is a finite sum of integrals of
types $J_{i}(\omega_{s})$, $i=0,1,2$, see (\ref{JAA}) -(\ref{JC}), it is
sufficient to show the following:
\begin{equation}
\lim_{s\rightarrow\alpha}\left(  1-q^{s-\alpha}\right)  ^{n}J_{i}(\omega
_{s})\geq0\text{, }i=0,1,2 \label{(i)}%
\end{equation}
and%
\begin{equation}
\lim_{s\rightarrow\alpha}\left(  1-q^{s-\alpha}\right)  ^{n}J_{0}(\omega
_{s})>0. \label{(ii)}%
\end{equation}

Let $\Delta$ be a cone in $\mathcal{L}_{n}(\alpha)$ spanned by $a_{i}$,
$i=1,\ldots,n$, with $\frac{\left\Vert a_{i}\right\Vert }{-d\left(
a_{i}\right)  }=\alpha$ for $i=1,\ldots,n$. Take $b$ in $X\left(
\mathcal{F}\right)  $ to be the origin of the chart $\left(  K_{\Delta}%
^{n},\sigma_{\Delta}\right)  $ corresponding to $\Delta$ and use formulas
(\ref{Ec1})-(\ref{Ec3}). Furthermore, we pick $\Phi$ in such a way that the
neighborhood $V_{0}$ of \ the origin where (\ref{Ec4AA})-(\ref{Ec5C}) are
valid be equal to $\left(  \mathfrak{p}^{e}R_{K}\right)  ^{n}$. Then by using
Lemma \ref{lemma3}, $J_{0}(\omega_{s})$ equals
\begin{align}
&  \frac{\left(  1-q^{-1}\right)  ^{n}\Phi\left(  \sigma\left(  b\right)
\right)  \left\vert \varepsilon\left(  b\right)  \right\vert _{K}%
^{s}\left\vert \eta\left(  b\right)  \right\vert _{K}q^{-e\left\{  \sum
_{j=1}^{n}d\left(  a_{j}\right)  \right\}  \left(  s-\alpha\right)  }}{\left(
1-q^{s-\alpha}\right)  ^{n}}\times\label{formula}\\
&  \left(
{\displaystyle\prod\limits_{j=1}^{n}}
\frac{1}{%
{\displaystyle\prod\limits_{\varsigma_{j}\neq1,\varsigma_{j}^{{\small d}%
\left(  a_{j}\right)  }=1}}
\left(  1-\varsigma_{j}q^{s-\alpha}\right)  }\right)  .\nonumber
\end{align}
Then
\begin{align}
\lim_{s\rightarrow\alpha}\left(  1-q^{s-\alpha}\right)  ^{n}J_{0}(\omega_{s})
&  =\label{formula1}\\
\frac{\left(  1-q^{-1}\right)  ^{n}\Phi\left(  \sigma\left(  b\right)
\right)  \left\vert \varepsilon\left(  b\right)  \right\vert _{K}^{\alpha
}\left\vert \eta\left(  b\right)  \right\vert _{K}}{%
{\displaystyle\prod\limits_{j=1}^{n}}
\left\vert d\left(  a_{j}\right)  \right\vert }  &  >0.\nonumber
\end{align}

We note that the previous limit does not depend on the branch of the complex
logarithm used to defined the complex power of $q$. By using a similar
reasoning, one verifies that
\[
\lim_{s\rightarrow\alpha}\left(  1-q^{s-\alpha}\right)  ^{n}J_{1}(\omega
_{s})=\text{ }\lim_{s\rightarrow\alpha}\left(  1-q^{s-\alpha}\right)
^{n}J_{2}(\omega_{s})=0.
\]
On the other hand, if $\Delta\notin\mathcal{L}_{n}(\alpha)$, then
$\lim_{s\rightarrow\alpha}\left(  1-q^{s-\alpha}\right)  ^{n}J_{i}(\omega
_{s})=0$,\ for $i=0,1,2$.
\end{proof}

\begin{lemma}
\label{prop3}Assume that $B\left(  \mathcal{F}\right)  \neq\emptyset$. Given
$l\in\mathbb{N}$, with $1\leq l\leq n$, define
\[
\mathcal{M}_{l}(\beta)=\left\{
\begin{array}
[c]{c}%
\Delta\in\mathcal{F};\Delta\text{ has exactly }l\text{ edges, }a_{k}%
\text{,satisfying }\\
\frac{\left\Vert a_{k}\right\Vert }{-d\left(  a_{k}\right)  }=\beta\text{, for
}k=1,\ldots,l.
\end{array}
\right\}
\]
If $\max_{l}\left\{  \mathcal{M}_{l}(\beta_{f})\neq\emptyset\right\}  =n$,
then $Z_{\bullet}\left(  \omega\right)  $ has a pole $s$ \ of multiplicity $n$
satisfying $\operatorname{Re}(s)=\beta$.
\end{lemma}

\begin{proof}
It is similar to the proof of Lemma \ref{prop2}.
\end{proof}

\subsection{Volumes of Tubes}

The classical local zeta functions attached to polynomials are connected with
the number of solutions of polynomial congruences. The local zeta functions
attached to Laurent polynomials are connected with the volumes of certain
tubes determined by the Laurent polynomial.

\begin{theorem}
\label{Th1AA}Let $f$ be a Laurent polynomial which is weakly non-degenerate
with respect to $\Gamma_{\infty}$. Set for $m\in\mathbb{N}\smallsetminus
\left\{  0\right\}  $,%
\[
V_{-m}\left(  f,\Phi\right)  :=vol\left(  \left\{  x\in\text{supp}\left(
\Phi\right)  \cap T^{n}\left(  K\right)  ;\left\vert f\left(  x\right)
\right\vert _{K}=q^{-m}\right\}  \right)
\]
and%
\[
V_{m}\left(  f,\Phi\right)  :=vol\left(  \left\{  x\in\text{supp}\left(
\Phi\right)  \cap T^{n}\left(  K\right)  ;\left\vert f\left(  x\right)
\right\vert _{K}=q^{m}\right\}  \right)  .
\]
Then the following assertions hold.

\noindent(i) Assume that $Z_{\bullet}\left(  \omega\right)  $ has at least one
pole with negative real part. Then for $m$ big enough, $V_{-m}\left(
f,\Phi\right)  $ has an asymptotic expansion of the form%
\[
V_{-m}\left(  f,\Phi\right)  =%
{\displaystyle\sum\limits_{\gamma}}
c_{m}\left(  \gamma,f\right)  m^{j_{\gamma}}q^{\gamma m}%
\]
where $\gamma$ runs through all of the poles of $Z_{\Phi}(s,\chi_{triv},f)$
for which $\operatorname{Re}\left(  \gamma\right)  \in B(\mathcal{F})$,
$j_{\gamma}\leq$ (the multiplicity of $\gamma$)$-1$, and the $c_{m}\left(
\gamma,f\right)  $ are complex constants. Furthermore
\[
V_{-m}\left(  f,\Phi\right)  \leq Cm^{n-1}q^{m\beta}\text{ for }m\geq0,
\]
where $C$ is a positive constant.

\noindent(ii) Assume that $Z_{\bullet}\left(  \omega\right)  $ has at least
one pole with positive real part. If $\left\vert f\right\vert _{K}$ is not
bounded on supp$\left(  \Phi\right)  \cap T^{n}\left(  K\right)  $, then for
$m$ big enough, $V_{m}\left(  f,\Phi\right)  $ has an asymptotic expansion of
the form%
\[
V_{m}\left(  f,\Phi\right)  =%
{\displaystyle\sum\limits_{\gamma}}
c_{m}\left(  \gamma,f\right)  m^{j_{\gamma}}q^{-\gamma m},
\]
where $\gamma$ runs through all of the poles of $Z_{\Phi}(s,\chi_{triv},f)$
for which $\operatorname{Re}\left(  \gamma\right)  \in A(\mathcal{F})$,
$j_{\gamma}\leq$ (the multiplicity of $\gamma$)$-1$, and the $c_{m}\left(
\gamma,f\right)  $ are complex constants and. Furthermore
\[
V_{m}\left(  f,\Phi\right)  \leq Cm^{n-1}q^{-m\alpha},\text{ for }m\geq0,
\]
where $C$ is a positive constant.
\end{theorem}

\begin{proof}
We first note that
\begin{align*}
Z_{\Phi}(s,\chi_{triv},f)  &  =%
{\textstyle\int\limits_{T^{n}\left(  K\right)  }}
\Phi\left(  x\right)  \left\vert f\left(  x\right)  \right\vert _{K}%
^{s}\left\vert dx\right\vert \text{ }\\
&  =%
{\textstyle\sum\limits_{m\in\mathbb{Z}}}
vol\left(  \left\{  x\in\text{supp}\left(  \Phi\right)  ;\left\vert f\left(
x\right)  \right\vert _{K}=q^{-m}\right\}  t^{m}\right)  \text{, with
}t:=q^{-s}\text{,}%
\end{align*}
for $\beta<\operatorname{Re}(s)<\alpha$. Now, the announced results follow
from Theorem \ref{Th1} by expanding $Z_{\Phi}(s,\chi_{triv},f)$ into partial
fractions over the complex numbers. Since two variables $t$, $t^{-1}$ are
involved in the calculations and since we will need this technique later, we
present here some details. For the sake of simplicity, we give the proof of
the case $n=2$, the generalization to arbitrary $n$ is straightforward.

For $m\in\mathbb{Z\smallsetminus}\left\{  0\right\}  $, we write $m=\left\vert
m\right\vert sgn\left(  m\right)  =\left\vert m\right\vert \left(
\pm1\right)  $. We also set $U_{f}:=\left\{  \varsigma\in\mathbb{C}%
:\varsigma^{f}=1\right\}  $ for $f\in\mathbb{N\smallsetminus}\left\{
0\right\}  $. By using the identity%
\[
1-q^{-e}t^{\pm f}=\left(  1-q^{\frac{-e}{f}}t^{\pm1}\right)
{\displaystyle\prod\limits_{\varsigma\in U_{f}\smallsetminus\left\{
1\right\}  }}
\text{ }\left(  1-q^{\frac{-e}{f}}\varsigma t^{\pm1}\right)  \text{, }%
e,f\in\mathbb{N\smallsetminus}\left\{  0\right\}  ,
\]
we have%
\[
\frac{1}{1-q^{-\left\Vert a_{k}\right\Vert }t^{d\left(  a_{k}\right)  }}%
=\frac{1}{1-q^{-\left\Vert a_{k}\right\Vert }t^{\left\vert d\left(
a_{k}\right)  \right\vert \left(  \pm1\right)  }}=%
{\displaystyle\sum\limits_{\varsigma\in U_{\left\vert d\left(  a_{k}\right)
\right\vert }}}
c_{\varsigma}\left(
{\displaystyle\sum\limits_{l=0}^{+\infty}}
q^{\frac{-\left\Vert a_{k}\right\Vert }{\left\vert d\left(  a_{k}\right)
\right\vert }l}\varsigma^{l}t^{\pm l}\right)
\]
for some constants $c_{\varsigma}\in\mathbb{C}$. Note that $\pm l=l\left\{
sgn(d\left(  a_{k}\right)  )\right\}  $. If $\frac{-\left\Vert a_{i}%
\right\Vert }{\left\vert d\left(  a_{i}\right)  \right\vert }\neq
\frac{-\left\Vert a_{j}\right\Vert }{\left\vert d\left(  a_{j}\right)
\right\vert }$, then%
\begin{align*}
\frac{1}{\left(  1-q^{-\left\Vert a_{i}\right\Vert }t^{d\left(  a_{i}\right)
}\right)  \left(  1-q^{-\left\Vert a_{j}\right\Vert }t^{d\left(  a_{j}\right)
}\right)  }  &  =%
{\displaystyle\sum\limits_{\varsigma\in U_{\left\vert d\left(  a_{i}\right)
\right\vert }}}
d_{\varsigma}\left(
{\displaystyle\sum\limits_{l=0}^{+\infty}}
q^{\frac{-\left\Vert a_{i}\right\Vert }{\left\vert d\left(  a_{i}\right)
\right\vert }l}\varsigma^{l}t^{\pm l}\right) \\
&  +%
{\displaystyle\sum\limits_{\varsigma\in U_{\left\vert d\left(  a_{j}\right)
\right\vert }}}
h_{\varsigma}\left(
{\displaystyle\sum\limits_{l=0}^{+\infty}}
q^{\frac{-\left\Vert a_{j}\right\Vert }{\left\vert d\left(  a_{j}\right)
\right\vert }l}\varsigma^{l}t^{\pm l}\right)
\end{align*}
for some constants $d_{\varsigma},h_{\varsigma}\in\mathbb{C}$. If
$\frac{-\left\Vert a_{i}\right\Vert }{\left\vert d\left(  a_{i}\right)
\right\vert }=\frac{-\left\Vert a_{j}\right\Vert }{\left\vert d\left(
a_{j}\right)  \right\vert }$, then%

\[
\frac{1}{\left(  1-q^{-\left\Vert a_{i}\right\Vert }t^{d\left(  a_{i}\right)
}\right)  \left(  1-q^{-\left\Vert a_{j}\right\Vert }t^{d\left(  a_{j}\right)
}\right)  }=%
{\displaystyle\sum\limits_{\varsigma\in U_{\left\vert d\left(  a_{i}\right)
\right\vert }\cap U_{\left\vert d\left(  a_{j}\right)  \right\vert }}}
\left\{  \text{ }\frac{d_{\varsigma}}{\left(  1-q^{\frac{-\left\Vert
a_{i}\right\Vert }{\left\vert d\left(  a_{i}\right)  \right\vert }}\varsigma
t^{\pm1}\right)  ^{2}}\right.
\]%
\[
\left.  +\frac{f_{\varsigma}}{1-q^{\frac{-\left\Vert a_{i}\right\Vert
}{\left\vert d\left(  a_{i}\right)  \right\vert }}\varsigma t^{\pm1}}\right\}
+%
{\displaystyle\sum\limits_{%
\begin{array}
[c]{c}%
\varsigma\in U_{\left\vert d\left(  a_{i}\right)  \right\vert }\\
\varsigma\notin U_{\left\vert d\left(  a_{i}\right)  \right\vert }\cap
U_{\left\vert d\left(  a_{j}\right)  \right\vert }%
\end{array}
}}
g_{\varsigma}\left(
{\displaystyle\sum\limits_{l=0}^{+\infty}}
q^{\frac{-\left\Vert a_{i}\right\Vert }{\left\vert d\left(  a_{i}\right)
\right\vert }l}\varsigma^{l}t^{\pm l}\right)
\]%
\[
+%
{\displaystyle\sum\limits_{%
\begin{array}
[c]{c}%
\varsigma\in U_{\left\vert d\left(  aj\right)  \right\vert }\\
\varsigma\notin U_{\left\vert d\left(  a_{i}\right)  \right\vert }\cap
U_{\left\vert d\left(  a_{j}\right)  \right\vert }%
\end{array}
}}
h_{\varsigma}\left(
{\displaystyle\sum\limits_{l=0}^{+\infty}}
q^{\frac{-\left\Vert a_{j}\right\Vert }{\left\vert d\left(  a_{j}\right)
\right\vert }l}\varsigma^{l}t^{\pm l}\right)
\]
for some constants $d_{\varsigma},f_{\varsigma},g_{\varsigma},h_{\varsigma}%
\in\mathbb{C}$. Note that%
\[
\frac{1}{\left(  1-q^{\frac{-\left\Vert a_{i}\right\Vert }{\left\vert d\left(
a_{i}\right)  \right\vert }}\varsigma t^{\pm1}\right)  ^{2}}=%
{\displaystyle\sum\limits_{l=0}^{+\infty}}
\left(  l+1\right)  q^{\frac{-\left\Vert a_{i}\right\Vert }{\left\vert
d\left(  a_{i}\right)  \right\vert }l}\varsigma^{l}t^{\pm l}.
\]
Therefore for $m$ big enough,
\begin{equation}
V_{-m}\left(  f,\Phi\right)  =%
{\displaystyle\sum\limits_{\gamma}}
c_{m}\left(  \gamma,f\right)  m^{j_{\gamma}}q^{\gamma m} \label{Iden6}%
\end{equation}
where $\gamma$ runs through all of the poles of $Z_{\Phi}(s,\chi_{triv},f)$
such that $\operatorname{Re}\left(  \gamma\right)  \in B(\mathcal{F})$,
$j_{\gamma}\leq$ (the multiplicity of $\gamma$)$-1$, and the $c(m,\gamma)$ are
complex constants. The first part follows from (\ref{Iden6}). The second part
is established in a similar form.
\end{proof}

\subsection{\label{Vanishing}Vanishing of local zeta functions}

\begin{theorem}
\label{Th1A}Let $f$ be a weakly non-degenerate Laurent polynomial satisfying
Hypothesis H1. There exists a constant $e\left(  \Phi\right)  \in\mathbb{N}$,
such that $Z_{\Phi}\left(  s,\chi,f\right)  =0$ unless $c\left(  \chi\right)
\leq$ $e\left(  \Phi\right)  $.
\end{theorem}

\begin{proof}
The proof follows from formulas (\ref{Ec1})-(\ref{Ec6}), Lemma \ref{lemma3},
by using the same argument given by Igusa for Theorem 8.4.1 in \cite{I2}.
\end{proof}

\section{\label{Sect_OS}Oscillatory Integrals}

In this section we extend Igusa's stationary phase method for $p$-adic
oscillatory integrals (\cite{I1}, \cite{I2}, \cite{D0}) to the case of
non-degenerate Laurent polynomials.

\subsection{\label{characters}Additive characters}

Given
\[
z=\sum_{n=n_{0}}^{\infty}z_{n}p^{n}\in\mathbb{Q}_{p}\text{, with }z_{n}%
\in\left\{  0,\ldots,p-1\right\}  \text{ and }z_{n_{0}}\neq0,
\]
we set
\[
\left\{  z\right\}  _{p}:=\left\{
\begin{array}
[c]{lll}%
0 & \text{if} & n_{0}\geq0\\
&  & \\
\sum_{n=n_{0}}^{-1}z_{n}p^{n} & \text{if} & n_{0}<0,
\end{array}
\right.
\]
\textit{the fractional part of }$z$. Then $\exp(2\pi\sqrt{-1}\left\{
z\right\}  _{p}),$ $z\in\mathbb{Q}_{p}$, is an additive character on
$\mathbb{Q}_{p}$ trivial on $\mathbb{Z}_{p}$ but not on $p^{-1}\mathbb{Z}_{p}$.

We recall that there exists an integer $d\geq0$ such that $Tr_{K/\mathbb{Q}%
_{p}}(z)\in\mathbb{Z}_{p}$ for $\left\vert z\right\vert _{K}\leq q^{d}$ but
$Tr_{K/\mathbb{Q}_{p}}(z_{0})\notin\mathbb{Z}_{p}$ for some $z_{0}$\ with
$\left\vert z_{0}\right\vert _{K}=q^{d+1}$. The integer $d$ is called
\textit{the exponent of the different} of $K/\mathbb{Q}_{p}$. It is known that
$d\geq e-1$, where $e$ is the ramification index of $K/\mathbb{Q}_{p}$, see
e.g. \cite[Chap. VIII, Corollary of Proposition 1]{W}. The additive character%
\[
\varkappa(z)=\exp(2\pi\sqrt{-1}\left\{  Tr_{K/\mathbb{Q}_{p}}(\mathfrak{p}%
^{-d}z)\right\}  _{p}),\text{ }z\in K\text{, }%
\]
is \textit{a standard character} of $K$, i.e. $\varkappa$ is trivial on
$R_{K}$ but not on $P_{K}^{-1}$. For our purposes, it is more convenient to
use
\[
\Psi(z)=\exp(2\pi\sqrt{-1}\left\{  Tr_{K/\mathbb{Q}_{p}}(z)\right\}
_{p}),\text{ }z\in K\text{, }%
\]
instead of $\varkappa(\cdot)$. This particular choice is due to the fact that
we use Denef's approach for estimating oscillatory integrals, see
\cite[Proposition 1.4.4]{D0}.

\subsection{Asymptotic expansion of oscillatory integrals}

Given $\Phi\in S(K^{n})$ and $f$ a Laurent polynomial as before, we define%
\[
E_{\Phi}\left(  z,f\right)  =E_{\Phi}\left(  z\right)  =%
{\displaystyle\int\limits_{T^{n}\left(  K\right)  }}
\Phi\left(  x\right)  \Psi\left(  zf\left(  x\right)  \right)  \left\vert
dx\right\vert ,
\]
for $z=u\mathfrak{p}^{-m}$, with $u\in R_{K}^{\times}$, and $m\in\mathbb{Z}$.
We call a such integral an \textit{oscillatory integral}.

Let Coeff$_{t^{k}}Z_{\Phi}(s,\chi,f)$ denote the coefficient $c_{k}$ in the
power expansion of $Z_{\Phi}(s,\chi,f)$ in the variable $t=q^{-s}.$

\begin{proposition}
\label{prop4}With the above notation,
\begin{align*}
E_{\Phi}\left(  u\pi^{-m}\right)   &  =Z_{\Phi}(0,\chi_{\text{triv}%
})+\text{Coeff}_{t^{m-1}}\frac{\left(  t-q\right)  Z_{\Phi}(s,\chi
_{\text{triv}})}{\left(  q-1\right)  \left(  1-t\right)  }+\\
&
{\displaystyle\sum\limits_{\chi\neq\chi_{\text{triv}}}}
g_{\chi^{-1}}\chi\left(  u\right)  \text{Coeff}_{t^{m-c\left(  \chi\right)  }%
}Z_{\Phi}(s,\chi),
\end{align*}
where $c\left(  \chi\right)  $\ denotes the conductor of $\chi$, and $g_{\chi
}$ denotes the Gaussian sum%
\[
g_{\chi}=\left(  q-1\right)  ^{-1}q^{1-c\left(  \chi\right)  }%
{\displaystyle\sum\limits_{v\in\left(  R_{K}/P_{K}^{c\left(  \chi\right)
}\right)  ^{\times}}}
\chi\left(  v\right)  \Psi\left(  v/\pi^{c\left(  \chi\right)  }\right)  .
\]

\end{proposition}

\begin{proof}
The proof is similar to the proof of Proposition 1.4.4 in \cite{D0}.
\end{proof}

\begin{theorem}
\label{Th2}Let $f$ be a Laurent polynomial which is weakly non-degenerate with
respect to $\Gamma_{\infty}$. Let $\mathcal{F}$ be a nontrivial simple fan
subordinated to $\Gamma_{\infty}$ as before. Assume that $f$ satisfies
Hypothesis H1. Then the following assertions hold.

\noindent(i) Assume that $Z_{\bullet}\left(  \omega\right)  $ has at least one
pole with negative real part. Then for $\left\vert z\right\vert _{K}$ big
enough $E_{\Phi}(z)$ is a finite $\mathbb{C}-$linear combination of functions
of the form
\[
\chi\left(  ac\text{ }z\right)  \left\vert z\right\vert _{K}^{\lambda}\left(
\ln_{q}\left\vert z\right\vert _{K}\right)  ^{j_{\lambda}}%
\]
with coefficients independent of $z$, and $\lambda\in\mathbb{C}$ a pole with
negative real part of $\left(  1-q^{-s-1}\right)  Z_{\Phi}(s,\chi
_{\text{triv}})$ or of $Z_{\Phi}(s,\chi)$, $\chi\neq\chi_{\text{triv}}$, and
with $j_{\lambda}\leq$(multiplicity of pole $\lambda$) $-1$. Moreover all the
poles $\lambda$,\ with negative real part, appear effectively in this linear combination.

\noindent(ii) Furthermore,
\[
\left\vert E_{\Phi}(z)\right\vert \leq C\left(  K\right)  \left\vert
z\right\vert _{K}^{\beta}\left(  \ln_{q}\left\vert z\right\vert _{K}\right)
^{n-1}%
\]
for $\left\vert z\right\vert _{K}$ big enough, where $C\left(  K\right)  $\ is
a positive constant.

\noindent(iii) Assume that $Z_{\bullet}\left(  \omega\right)  $ has at least
one pole with positive real part. Then for $\left\vert z\right\vert _{K}$
small enough $E_{\Phi}(z)-Z_{\Phi}(0,\chi_{\text{triv}})$ is a finite
$\mathbb{C}-$linear combination of functions of the form
\[
\chi\left(  ac\text{ }z\right)  \left\vert z\right\vert _{K}^{\lambda}\left(
\ln_{q}\left\vert z\right\vert _{K}\right)  ^{j_{\lambda}}%
\]
with coefficients independent of $z$, and $\lambda\in\mathbb{C}$ a pole with
positive real part of $Z_{\Phi}(s,\chi)$, and with $j_{\lambda}\leq
$(multiplicity of pole $\lambda$)$-1$. Moreover all the poles $\lambda$,\ with
positive real part, appear effectively in this linear combination.

\noindent(iv) Furthermore,
\[
\left\vert E_{\Phi}(z)-Z_{\Phi}(0,\chi_{\text{triv}})\right\vert \leq C\left(
K\right)  \left\vert z\right\vert _{K}^{\alpha}\left(  \ln_{q}\left\vert
z\right\vert _{K}\right)  ^{n-1}%
\]
for $\left\vert z\right\vert _{K}$ small enough, where \ $C\left(  K\right)
$\ is a positive constant.
\end{theorem}

\begin{proof}
The results follow from Theorem \ref{Th1},\ Proposition \ref{prop4} \ and
Theorem \ref{Th1A}, by writing $Z_{\Phi}(s,\chi)$ in partial fractions, as in
the proof of Theorem \ref{Th1AA}.
\end{proof}

In general $E_{\Phi}(z,f)$ cannot be expressed as a finite sum of exponential
sums mod $\mathfrak{p}^{m}$. The following result shows that, under additional
hypotheses, $E_{\Phi}(z,f)$ becomes an exponential sum mod $\mathfrak{p}^{m}$.

\begin{corollary}
\label{cor1}Let $f\left(  x\right)  =\frac{\widehat{f}\left(  x\right)  }{%
{\textstyle\prod\nolimits_{i=1}^{r}}
x_{i}^{d_{i}}}$, $1\leq r\leq n-1$, be a non-degenerate Laurent polynomial as
before. Set
\[
S_{m}\left(  f\right)  :=q^{-mn}%
{\displaystyle\sum\limits_{x\in\left(  R_{K}^{\times}/P_{K}^{m}\right)
^{r}\times\left(  R_{K}/P_{K}^{m}\right)  ^{n-r}}}
\Psi\left(  zf\left(  x\right)  \right)  ,
\]
where $z=u\mathfrak{p}^{-m}$, with $u\in R_{K}^{\times}$ and $m\geq1$. Then,
for $m$ big enough,
\[
\left\vert S_{m}\left(  f\right)  \right\vert \leq Cm^{n-1}q^{m\beta}.
\]

\end{corollary}

\begin{proof}
Take $\Phi$ to be the characteristic function of $\left(  R_{K}^{\times
}\right)  ^{r}\times R_{K}^{n-r}$, then $E_{\Phi}(z,f)\allowbreak=S_{m}\left(
f\right)  $. Now the result follows from Theorem \ref{Th2} (i).
\end{proof}


\begin{thebibliography}{99}                                                                                               %
\bibitem {A-S2}Adolphson Alan and Sperber Steven, Exponential sums and Newton
polyhedra: cohomology and estimates, Ann. of Math. (2) 130 (1989), 367--406.

\bibitem {AVG}Arnold V. I., Gussein-Zade S. M., Varchenko A. N., Singularités
des applications différentiables, Vol II, Éditions Mir, Moscou, 1986.

\bibitem {D0}Denef J., Report on Igusa's Local Zeta Function, Séminaire
Bourbaki 43 (1990-1991), exp. 741; Astérisque 201-202-203 (1991), 359-386.
Available at http://www.wis.kuleuven.ac.be/algebra/denef.html.

\bibitem {D3}Denef J., Poles \ of $p$-adic complex powers \ and Newton
polyhedra, Nieuw. Arch. Wisk. 13 (1995), 289-295.

\bibitem {D-H}Denef J. and \ Hoornaert K., Newton polyhedra and Igusa's local
zeta function, J. Number Theory 89 (2001), 31-64.

\bibitem {D-L}Denef J. and Loeser F., Weights of exponential sums,
intersection cohomology, and Newton polyhedra, Invent. Math. 106 (1991), no.
2, 275--294.

\bibitem {D-S}Denef J. and Sargos P., Polyèdre de Newton et distribution
$f_{+}^{s}$. I, J. Analyse Math., 53 (1989), 201--218.

\bibitem {D-Sp}Denef J., Sperber S., Exponential sums mod $p^{n}$ and Newton
polyhedra. A tribute to Maurice Boffa.Bull. Belg. Math. Soc. Simon Stevin
2001, suppl., 55--63.

\bibitem {De-Van}Denef J., van den Dries L., $p$-adic and real subanalytic
sets, Ann. of Math. (2) 128 (1988), no. 1, 79--138.

\bibitem {Ew}Ewald Günter, Combinatorial convexity and algebraic geometry.
Graduate Texts in Mathematics, 168. Springer-Verlag, New York, 1996.

\bibitem {I1}Igusa J-Ii, Forms of higher degree, Tata Institute of Fundamental
Research Lectures on Mathematics and Physics, 59. Tata Institute of
Fundamental Research, Bombay; by the Narosa Publishing House, New Delhi, 1978.

\bibitem {I2}Igusa J.-I., An introduction \ to the theory of local zeta
functions, AMS/IP Studies in Advanced Mathematics, 2000.

\bibitem {K-M-S}Kempf G., Knudsen F., Mumford D., Saint-Donat B., Toroidal
embeddings, Lectures notes in Mathematics vol. 339, Springer-Verlag, 1973.

\bibitem {KH1}Khovanskii A. G., Newton polyhedra, and toroidal varieties,
Functional Anal. Appl. 11 (1977), no. 4, 289--296 (1978).

\bibitem {KH2}Khovanskii A. G., Newton polyhedra (resolution of
singularities). (Russian) Current problems in mathematics, Vol. 22, 207--239,
Itogi Nauki i Tekhniki, Akad. Nauk SSSR, Vsesoyuz. Inst. Nauchn. i Tekhn.
Inform., Moscow, 1983.

\bibitem {L-M}Lichtin Ben, Meuser Diane, Poles of a local zeta function and
Newton polygons, Compositio Math. 55 (1985), no. 3, 313--332.

\bibitem {O}Oka Mutsuo, Non-degenerate Complete Intersection Singularity.
Actualités Mathématiques. [Current Mathematical Topics] Hermann, Paris, 1997.

\bibitem {Ser}Serre Jean-Pierre, Lie Algebras and Lie Groups. W. A. Benjamin,
Inc., New York, Amsterdam, 1968.

\bibitem {Var1}Varchenko A., Newton polyhedra and estimation of oscillating
integrals, \ Funct. Anal. Appl. 10 (1976), 175-196.

\bibitem {V-Z}Veys W. and Zúñiga-Galindo W. A., Zeta functions associated with
polynomial mappings, log-principalization,\ Trans. Amer. Math. Soc. 360
(2008), 2205-2227.

\bibitem {W}Weil A., Basic Number Theory, Springer-Verlag, Berlin, 1967.

\bibitem {Zie}Ziegler Günter M., Lectures on polytopes. Graduate Texts in
Mathematics, 152. Springer-Verlag, New York, 1994.

\bibitem {Z1}Zuniga-Galindo W.A., Local zeta functions and Newton polyhedra,
Nagoya Math. J., 172 (2003), 31-58.

\bibitem {Z2}Zúñiga-Galindo W. A., Local zeta functions supported on analytic
submanifolds and Newton polyhedra, Int. Math. Res. Not. IMRN 2009, no. 15, 2855--2898.
\end{thebibliography}
\end{document}